\documentclass{amsart}
\usepackage[LGR,T1]{fontenc}
\usepackage[latin9]{inputenc}
\usepackage{verbatim}
\usepackage{amstext}
\usepackage{amsthm}
\usepackage{amssymb}

\makeatletter

\DeclareRobustCommand{\greektext}{%
  \fontencoding{LGR}\selectfont\def\encodingdefault{LGR}}
\DeclareRobustCommand{\textgreek}[1]{\leavevmode{\greektext #1}}
\ProvideTextCommand{\~}{LGR}[1]{\char126#1}

\numberwithin{equation}{section}
\numberwithin{figure}{section}
\theoremstyle{plain}
\newtheorem{thm}{\protect\theoremname}
\theoremstyle{remark}
\newtheorem{rem}[thm]{\protect\remarkname}
\theoremstyle{plain}
\newtheorem{prop}[thm]{\protect\propositionname}
\theoremstyle{plain}
\newtheorem{lem}[thm]{\protect\lemmaname}
\theoremstyle{plain}
\newtheorem{cor}[thm]{\protect\corollaryname}

\@ifundefined{date}{}{\date{}}
\providecommand{\lemmaname}{Lemma}
\providecommand{\remarkname}{Remark}
\providecommand{\theoremname}{Theorem}

\makeatother

\providecommand{\corollaryname}{Corollary}
\providecommand{\lemmaname}{Lemma}
\providecommand{\propositionname}{Proposition}
\providecommand{\remarkname}{Remark}
\providecommand{\theoremname}{Theorem}

\begin{document}
\title{Local and global analyticity for $\mu$-Camassa-Holm equations}
\author{Hideshi Yamane}
\curraddr{{\small{}{}Department of Mathematical Sciences, Kwansei Gakuin University
}\\
 {\small{}{}Gakuen 2-1 Sanda, Hyogo 669-1337, Japan}}
\email{{\small{}{}yamane@kwansei.ac.jp}}
\thanks{This work was partially supported by JSPS KAKENHI Grant Number 26400127.}
\begin{abstract}
We solve Cauchy problems for some $\mu$-Camassa-Holm integro-partial
differential equations in the analytic category. The equations to
be considered are $\mu$CH of Khesin-Lenells-Misio\l{}ek, $\mu$DP
of Lenells-Misio\l{}ek-Ti\u{g}lay, the higher-order $\mu$CH of Wang-Li-Qiao
and the non-quasilinear version of Qu-Fu-Liu. We prove the unique
local solvability of the Cauchy problems and provide an estimate of
the lifespan of the solutions. Moreover, we show the existence of
a unique global-in-time analytic solution for $\mu$CH, $\mu$DP and
the higher-order $\mu$CH. The present work is the first result of
such a global nature for these equations.  
\end{abstract}


\maketitle
\markboth{Hideshi Yamane}{$\mu$-Camassa-Holm equations}\tableofcontents

\section*{Introduction}

We consider a functional equation 
\[
\mu(u_{t})-u_{txx}=-2\mu(u)u_{x}+2u_{x}u_{xx}+uu_{xxx},\;x\in S^{1}=\mathbb{R}/\mathbb{Z},
\]
called the $\mu$-Camassa-Holm equation ($\mu$CH), and its variants
in the complex-analytic or real-analytic category. Here $\mu(u)=\int_{S^{1}}u\,dx$.
Multiplying by the inverse of $\mu-\partial_{x}^{2}$, we get an evolution
equation
\[
u_{t}+uu_{x}+\partial_{x}(\mu-\partial_{x}^{2})^{-1}\left[2\mu(u)u+\frac{1}{2}u_{x}^{2}\right]=0.
\]
It motivates one to consider Cauchy problems, not only in Sobolev
spaces but also in spaces of analytic functions. In the latter case,
the solutions are  analytic in both $t$ and $x$. Recall that solutions
to the KdV equation can be analytic in $x$ but not in $t$. This
is because it is `not Kowalevskian', which means the first-order derivative
$u_{t}$ equals a quantity involving higher derivatives. Our evolution
equation mentioned above is `Kowalevskian' in a generalized sense
due to the presence of the negative order pseudodifferential operator
$(\mu-\partial_{x}^{2})^{-1}$.

Because of the nonlocal nature of $(\mu-\partial_{x}^{2})^{-1}$,
our considerations are always global in $x$. So we will work with
the Sobolev space $H^{m}(S^{1})$ or $A(\delta)$, the space of analytic
functions on $S^{1}\ni x$ which admits analytic continuation to $|y|<\delta$.
On the other hand, we can work either locally or globally in $t$.
Our local study will be given in Section \ref{sec:Local-in-time solutions}
and Appendix. It is based on the Ovsyannikov theorem used in \cite{BHP}
and \cite{HM}. It is a kind of abstract Cauchy-Kowalevsky theorem
about a scale of Banach spaces and enables us to obtain local-in-time
solutions which are analytic in both $t$ and $x$. Our global study
will be given in Sections \ref{sec:Global-in-time-solutions} and
\ref{sec:Global-in-time-solutions: higher}. Since global-in-time
solutions are known to exist in Sobolev spaces, what remains to be
done here is to prove their analyticity. We carry out this task by
using the method of \cite{Kato-Masuda} following \cite{BHPglobal}.
In the final part of the proof, we quote a result in \cite{Komatsu},
which gives a useful criterion of real analyticity.

Now we explain some background and history. In the course of it, we
will introduce some equations that will be studied in the present
paper. All the equations mentioned below are integrable in some sense.

The original Camassa-Holm equation
\begin{equation}
u_{t}-u_{txx}=-3uu_{x}+2u_{x}u_{xx}+uu_{xxx}\label{eq:CHoriginal}
\end{equation}
was introduced in \cite{CamassaHolm} (shallow water wave) and in
\cite{FokasFuchssteiner} (hereditary symmetries). It is known to
be completely integrable and admits peaked soliton (peakon) solutions.
The Cauchy problem for this equation can be formulated by introducing
a pseudodifferential operator. Indeed, \eqref{eq:CHoriginal} can
be written in the form 
\begin{equation}
u_{t}+uu_{x}+\partial_{x}(1-\partial_{x}^{2})^{-1}\left[u^{2}+\frac{1}{2}u_{x}^{2}\right]=0.\label{eq:CH}
\end{equation}
Since \eqref{eq:CH} is Kowalevskian in a generalized sense, it is
natural to solve this equation in the analytic setting as in the classical
Cauchy-Kowalevsky theorem. In \cite{BHP}, the authors introduced
a kind of Sobolev spaces with exponential weights consisting of holomorphic
functions in a strip of the type $|y|<\text{const.}$ Since these
spaces form a scale of Banach spaces, an Ovsyannikov type argument
can be applicable. It leads to the unique solvability and an estimate
of the lifespan of the solution in the periodic and non-periodic cases.

There are a lot of works about solutions of the Cauchy problem for
\eqref{eq:CHoriginal} or \eqref{eq:CH} in Sobolev spaces. See the
references in \cite{BHP} and \cite{ZY}. Local well-posedness and
blowup mechanism are major topics. In \cite{BHP}, the local unique
solvability in the analytic category was proved. Moreover, there is
a result about the global-in-time solvability in \cite{BHPglobal}.
Indeed, according to \cite{BHPglobal}, if the initial value is in
$H^{s}(\mathbb{R}),s>5/2$, and the McKean quantity $m_{0}=(1-\partial_{x}^{2})u_{0}$
does not change sign, then the Cauchy problem for (a generalization
of) \eqref{eq:CH} has a unique global-in-time solution $u\in\mathcal{C}([0,\infty);H^{s}(\mathbb{R}))$.
See \cite{McKean} for the necessity and sufficiency of the no-change-of-sign
condition. Moreover, in \cite{BHPglobal}, it is proved that this
solution is analytic in both $t$ and $x$ if the initial value is
in the space of analytic functions mentioned above. In the present
paper, we follow \cite{BHP} for local theory and the analyticity
part of \cite{BHPglobal} for global theory.

In \cite{KLM}, the $\mu$-version of \eqref{eq:CHoriginal}, namely
\begin{equation}
\mu(u_{t})-u_{txx}=-2\mu(u)u_{x}+2u_{x}u_{xx}+uu_{xxx},\;x\in S^{1}=\mathbb{R}/\mathbb{Z},\label{eq:muCH}
\end{equation}
was introduced. The authors call this equation $\mu$HS (HS is for
Hunter-Saxton), while it is called $\mu$CH in \cite{LMT}. We have
$\mu(u_{t})=0$, but we keep $\mu(u_{t})$ because this formulation
facilitates later calculation. The interest of \eqref{eq:muCH} lies,
for example, in the fact that it describes evolution of rotators in
liquid crystals with external magnetic field and self-interaction,
and it is related to the diffeomorphism groups of the circle with
a natural metric. Set 
\[
A(\varphi)=\mu(\varphi)-\varphi_{xx}.
\]
Then it is invertible for a suitable choice of function spaces and
commutes with $\partial_{x}$. The equation  \eqref{eq:muCH} can
be written in the following form (\cite[(5.1)]{KLM}): 
\begin{equation}
u_{t}+uu_{x}+\partial_{x}A^{-1}\left[2\mu(u)u+\frac{1}{2}u_{x}^{2}\right]=0.\label{eq:muCH2}
\end{equation}
In \cite{KLM}, the local well-posedness and the global existence
in Sobolev spaces is demonstrated. In the global problem, the $\mu$-McKean
quantity $(\mu-\partial_{x}^{2})u_{0}(x)$ is assumed to be free from
change of sign.

There are similar $\mu$-equations. In \cite[(5.3)]{LMT}, the following
equation, called $\mu$DP (DP is for Degasperis-Procesi), was introduced:
\begin{equation}
u_{t}+uu_{x}+\partial_{x}A^{-1}\left[3\mu(u)u\right]=0.\label{eq:muDP}
\end{equation}
The local well-posedness in $H^{s}(S^{1}),s>3/2$, and the global
existence in $H^{s}(S^{1}),s>3$, was proved in \cite{LMT}.

In \cite{CHK}, a family of higher-order Camassa-Holm equations depending
on $k=2,3,\dots$ was introduced. It is related to diffeomorphisms
of the unit circle. In \cite{WLQ}, the $\mu$-version of the $k=2$
case 
\begin{align}
 & u_{t}+uu_{x}+\partial_{x}B^{-1}\left[2\mu(u)u+\frac{1}{2}u_{x}^{2}-3u_{x}u_{xxx}-\frac{7}{2}u_{xx}^{2}\right]=0,\label{eq:higher}\\
 & B(\varphi)=\mu(\varphi)+(-\partial_{x}^{2}+\partial_{x}^{4})\varphi,\nonumber 
\end{align}
was formulated. The local well-posedness and the global existence
in $H^{s}(S^{1}),s>7/2$, was proved in \cite{WLQ}. Notice that the
no-change-of-sign condition is not imposed in this work. In \cite{WLQ2},
a different version 
\begin{equation}
u_{t}+uu_{x}+\partial_{x}A^{-2}\left[2\mu(u)u+\frac{1}{2}u_{x}^{2}-3u_{x}u_{xxx}-\frac{7}{2}u_{xx}^{2}\right]=0\label{eq:higher-1}
\end{equation}
is studied. The global existence is proved there.

The modified $\mu$-Camassa-Holm equation (modified $\mu$CH) with
non-quasilinear terms 
\begin{equation}
u_{t}+2\mu(u)uu_{x}-\frac{1}{3}u_{x}^{3}+\partial_{x}A^{-1}\left[2\mu^{2}(u)u+\mu(u)u_{x}^{2}+\gamma u\right]+\frac{1}{3}\mu(u_{x}^{3})=0\label{eq:modifiedmuCH2}
\end{equation}
was introduced in \cite[(3.2)]{QFL} ($\gamma=0$) and \cite[(2.7)]{QFL},
\cite{FuYing}. The global existence in Sobolev spaces remains open,
as opposed to that of the other equations mentioned above, so it is
not possible to show the global existence of analytic solutions by
using the method in the present article. Because of this exceptional
nature of the equation, we treat it in Appendix separately from the
others. Notice that local theory for \eqref{eq:modifiedmuCH2} is
developed in a certain space of analytic functions as well as in Besov
spaces in \cite{FuYing}.

The outline of this article is as follows. In Section \ref{sec:Function-spaces},
we introduce some function spaces and operators and investigate their
properties. In Section \ref{sec:Local-in-time solutions}, we prove
the local existence of analytic solutions of \eqref{eq:muCH2}, \eqref{eq:muDP},
\eqref{eq:higher} and \eqref{eq:higher-1}. These results are used
in the proofs of the global existence theorems in Sections \ref{sec:Global-in-time-solutions}
and \ref{sec:Global-in-time-solutions: higher} about \eqref{eq:muCH2},
\eqref{eq:muDP} and \eqref{eq:higher}. In Appendix, the local existence
of analytic solutions of \eqref{eq:modifiedmuCH2} is proved.

\section{Function spaces and operators\label{sec:Function-spaces}}

In the present paper, $L^{2}(S^{1})$ consists of \emph{real-valued}
square-integrable functions on $S^{1}=\mathbb{R}/\mathbb{Z}$. We
sometimes identify an element of it with a function on $\mathbb{R}$
with period 1. For a function on $S^{1}$, we set $\hat{\varphi}(k)=\int_{S^{1}}\varphi(x)e^{-2k\pi ix}\,dx$.
We introduce a family of Hilbert spaces $G^{\delta,s}$ ($\delta\ge0,s\ge0$)
by 
\[
G^{\delta,s}=\left\{ \varphi\in L^{2}(S^{1});\,\|\varphi\|_{\delta,s}<\infty\right\} ,\;\|\varphi\|_{\delta,s}^{2}=\sum_{k\in\mathbb{Z}}\langle k\rangle^{2s}e^{4\pi\delta|k|}|\hat{\varphi}(k)|^{2},
\]
where $\langle k\rangle=(1+k^{2})^{1/2}$ (Japanese bracket). Notice
that our definition is not exactly the same as that in \cite{BHP,BHPpower}.
In particular, the base space is $\mathbb{T}=\mathbb{R}/(2\pi\mathbb{Z})$
in \cite{BHP,BHPpower}. It is easy to see that we have continuous
injections $G^{\delta,s}\to G^{\delta',s'}$ and $G^{\delta,s}\to G^{\delta,s'}$
if $0\le\delta'<\delta,0\le s'<s$. Their norms are 1. We have a continuous
injection $G^{\delta,s}\to G^{\delta',s'}$ under the weaker assumption
$0\le\delta'<\delta$. We recover the usual Sobolev spaces
\[
H^{s}=G^{0,s},\;H^{\infty}=\cap_{s\ge0}H^{s}
\]
and we set $\|\varphi\|_{s}=\|\varphi\|_{0,s}$. Notice that $\|\cdot\|_{0}$
is the $L^{2}$ norm. The corresponding inner product is denoted by
$\langle\cdot,\cdot\rangle_{0}$. Set $\Lambda^{2}=1-(2\pi)^{-2}\partial_{x}^{2}$.
Then we have 
\begin{equation}
\|\varphi\|_{2}^{2}=\|\Lambda^{2}\varphi\|_{0}^{2}=\|\varphi\|_{0}^{2}+\frac{1}{2\pi^{2}}\|\varphi'\|_{0}^{2}+\frac{1}{(2\pi)^{4}}\|\varphi''\|_{0}^{2},\label{eq:H2norm}
\end{equation}
because $\int_{S^{1}}\varphi\varphi''\,dx=-\int_{S^{1}}(\varphi')^{2}\,dx$.

When $\delta>0$, set 
\begin{align*}
S(\delta) & =\left\{ z=x+iy\in\mathbb{C};\,|y|<\delta)\right\} ,\\
A(\delta) & =\left\{ f\colon S^{1}=\mathbb{R}/\mathbb{Z}\to\mathbb{R};\,f\,\text{has an analytic continuation to }S(\delta)\right\} .
\end{align*}

\begin{rem}
If a function is analytic on $S^{1}$, then it belong to $A(\delta)$
for some $\delta>0$. See Proposition \ref{prop:analyticity} below.
Notice that an analogous statement does not hold true if $S^{1}$
is replaced with $\mathbb{R}$.
\end{rem}

We identify an element of $A(\delta)$ with its analytic continuation.
For $f\in A(\delta)$, set
\[
\|f\|_{(\sigma,s)}^{2}=\sum_{j=0}^{\infty}\frac{e^{4\pi j\sigma}}{j!^{2}}\|f^{(j)}\|_{s}^{2}\qquad(e^{2\pi\sigma}<\delta).
\]
This norm will be used in the global theory. Do not confuse $\|\cdot\|_{(\sigma,s)}^{2}$
with $\|\cdot\|_{\delta,s}$.
\begin{prop}
\label{prop:embed G A}For $\delta>0,s\ge0$, the space $G^{\delta,s}$
is continuously embedded in $A(\delta)$.
\end{prop}

\begin{proof}
Assume $\varphi\in G^{\delta,s}$. The series $\varphi(z)=\sum\hat{\varphi}(k)e^{2k\pi iz}=\sum\hat{\varphi}(k)e^{-2k\pi y}e^{2k\pi ix}$
converges locally uniformly in $|y|<\delta$, because 
\begin{align*}
|\varphi(z)|^{2} & \le\sum\langle k\rangle^{2s}e^{4\pi\delta|k|}|\hat{\varphi}(k)|^{2}\times\sum\langle k\rangle^{-2s}e^{-4\pi\delta|k|}e^{-4k\pi y}\\
 & =\|\varphi\|_{\delta,s}^{2}\sum\langle k\rangle^{-2s}e^{-4\pi\delta|k|}e^{-4k\pi y}.
\end{align*}
Therefore $G^{\delta,s}\subset A(\delta)$ and this embedding is continuous
from $G^{\delta,s}$ to $A(\delta)$ with the topology (i) in Corollary
\ref{cor:Frechet}.
\end{proof}
\begin{rem}
\label{rem:Sobolev} Proposition \ref{prop:embed G A} is a variant
of the Sobolev embedding theorem: if $s>1/2$, then there is a continuous
embedding $H^{s}=G^{0,s}\to\mathcal{C}^{0}(S^{1})$ as is proved by
d
\[
|\varphi(x)|^{2}\le\left|\sum\hat{\varphi}(k)e^{2k\pi ix}\right|^{2}\le\sum\langle k\rangle^{2s}|\hat{\varphi}(k)|^{2}\times\sum\langle k\rangle^{-2s}|e^{2k\pi ix}|^{2}\le\|\varphi\|_{s}^{2}\sum\langle k\rangle^{-2s}.
\]
\end{rem}

\begin{prop}
\label{prop:analyticity}If $\varphi$ is a real-analytic function
on $S^{1}$, then there exists $\delta>0$ such that $\varphi\in G^{\delta,s}$
for any $s$. More precisely, if $\varphi\in A(\delta)$, then $\varphi\in G^{\delta',s}$
for any $\delta'\in]0,\delta[$ and any $s$.
\end{prop}

\begin{proof}
(This proposition is given in \cite{BHP} without proof.) By the periodicity,
contour deformation gives
\[
\hat{\varphi}(k)=\int_{0}^{1}\varphi\left(x\pm\delta'i\right)\exp\left(-2k\pi i\left[x\pm\delta'i\right]\right)\,dx.
\]
Set 
\[
a_{k}=\int_{0}^{1}\varphi\left(x-\mathrm{sgn}(k)\delta'i\right)e^{-2k\pi ix}\,dx=e^{2\pi|k|\delta'}\hat{\varphi}(k),\;k\ne0.
\]
These are the Fourier coefficients of $\varphi(x\pm i\delta')$ and
we have $\sum_{k\in\mathbb{Z}\setminus\left\{ 0\right\} }\langle k\rangle^{2s}|a_{k}|^{2}<\infty$
for any $s$. Therefore the series $\sum_{k\in\mathbb{Z}}\langle k\rangle^{2s}e^{4\pi\delta'|k|}|\hat{\varphi}(k)|^{2}$
converges.
\end{proof}
\begin{prop}
\label{prop:multiplication} We have the following three estimates
about products of functions:

(i) Assume $s>1/2,\delta\ge0$. Then $G^{\delta,s}$ is closed under
pointwise multiplication and we have 
\[
\|\varphi\psi\|_{\delta,s}\le c_{s}\|\varphi\|_{\delta,s}\|\psi\|_{\delta,s},\quad c_{s}=\left[2(1+s^{2s})\sum_{k=0}^{\infty}\langle k\rangle^{-2s}\right].
\]
\begin{comment}
$c_{s}=2^{2s+1}\sum_{\ell\in\mathbb{Z}}\langle\ell\rangle^{-2s}$
\end{comment}

(ii) There exists a positive constant $d_{s}$ such that we have $\|\varphi\psi\|_{0}\le d_{s}\|\varphi\|_{0}\|\psi\|_{s}$
for any $\varphi\in H^{0}=L^{2}(S^{1})$ and any $\psi\in H^{s}\,(s\ge1)$.

(iii) There exists a constant $\gamma$ such that we have $\|\varphi\psi\|_{2}\le\gamma\left(\|\varphi\|_{1}\|\psi\|_{2}+\|\varphi\|_{2}\|\psi\|_{1}\right)$
for any $\varphi,\psi\in H^{2}$.
\end{prop}

\begin{proof}
The proof of (i) is given in \cite{BHPpower}. Although different
formulations are used in \cite{BHPpower} and the present article,
the constant $c_{s}$ is the same. This is because $\widehat{\varphi\psi}(k)=\sum_{n}\hat{\varphi}(n)\hat{\psi}(k-n)$
holds in either situations. The difference of the base spaces are
offset by that of conventions, namely the presence or the absence
of the $1/(2\pi)$ factor in the definition of the Fourier coefficients.
The other three estimates follow from the boundedness of $H^{1}\hookrightarrow C^{0}(S^{1})$
(and the Leibniz rule).
\end{proof}
\begin{rem}
A better estimate than Proposition \ref{prop:multiplication} (iii)
can be found in \cite{Kato-Ponce}, which implies that $\|\cdot\|_{1}$
can be replaced with $\|\cdot\|_{L^{\infty}}$. It is used in \cite{WLQ},
but (iii) is good enough in the present paper.
\end{rem}

\begin{prop}
\label{prop:partial_x}(\cite[Lemma 2]{BHP}) If $0\le\delta'<\delta,s\ge0$
and $\varphi\in G^{\delta,s}$, then 
\begin{align*}
\|\varphi_{x}\|_{\delta',s} & \le\frac{e^{-1}}{\delta-\delta'}\|\varphi\|_{\delta,s},\\
\|\varphi_{x}\|_{\delta,s} & \le2\pi\|\varphi\|_{\delta,s+1}.
\end{align*}
A derivative can be estimated in two ways: `larger $\delta$' or `larger
$s$'.
\end{prop}

\begin{proof}
The second inequality is easy to prove. We give a proof of the first
in order to clarify that the assumption $\delta\le1$ in \cite{BHP,BHPpower}
is superfluous. The present author thinks that the authors of \cite{BHP,BHPpower}
wrote $\delta\le1$ not because they really needed it for the omitted
proof but for the sole reason that they were interested only in $0<\delta\le1$.

Set $f(x)=x^{2}\exp[2(-\delta+\delta')x],\,x\ge0$. We have 
\begin{align*}
\|\varphi\|_{\delta,s}^{2} & =\sum_{k\in\mathbb{Z}}\langle k\rangle^{2s}e^{4\pi\delta|k|}|\hat{\varphi}(k)|^{2},\\
\|\varphi_{x}\|_{\delta',s}^{2} & =\sum_{k\in\mathbb{Z}}\langle k\rangle^{2s}e^{4\pi\delta'|k|}(2k\pi)^{2}|\hat{\varphi}(k)|^{2}\\
 & =\sum_{k\in\mathbb{Z}}\langle k\rangle^{2s}e^{4\pi\delta|k|}f(2\pi|k|)|\hat{\varphi}(k)|^{2}.
\end{align*}
Since $0\le f(x)\le(e^{-1}/(\delta-\delta'))^{2}$, we get $\|\varphi_{x}\|_{\delta',s}^{2}\le(e^{-1}/(\delta-\delta'))^{2}\|\varphi\|_{\delta,s}^{2}$.
\end{proof}
We set $A(\varphi)=\mu(\varphi)-\varphi_{xx}$, $B(\varphi)=\mu(\varphi)+(-\partial_{x}^{2}+\partial_{x}^{4})\varphi$.
For $\varphi=\sum_{k\in\mathbb{Z}}a_{k}e^{2k\pi ix}\in G^{\delta,s}$,
we have 
\begin{align*}
\mu(\varphi) & =a_{0},\\
A(\varphi) & =a_{0}+\sum_{k\ne0}(2\pi k)^{2}a_{k}e^{2k\pi ix},\\
A^{-1}(\varphi) & =a_{0}+\sum_{k\ne0}\dfrac{a_{k}}{(2\pi k)^{2}}e^{2k\pi ix},\\
B(\varphi) & =a_{0}+\sum_{k\ne0}[(2\pi k)^{2}+(2\pi k)^{4}]a_{k}e^{2k\pi ix},\\
B^{-1}(\varphi) & =a_{0}+\sum_{k\ne0}\dfrac{a_{k}}{(2\pi k)^{2}+(2\pi k)^{4}}e^{2k\pi ix}.
\end{align*}
It follows that $A$ is a bounded operator from $G^{\delta,s+2}$
to $G^{\delta,s}$. It is a bijection and its inverse $A^{-1}$ is
a pseudodifferential operator of order $-2$. Therefore it is bounded
from $G^{\delta,s}$ to $G^{\delta,s+2}$. On the other hand, $B^{-1}$
is a pseudodifferential operator of order $-4$. Notice that $A^{-1}$
and $B^{-1}$ commute with $\partial_{x}$.\\
The following proposition is easy to prove.
\begin{prop}
\label{prop:smoothing}We have 
\begin{align*}
 & \left|\mu(\varphi)\right|\le\|\varphi\|_{\delta,s},\\
 & \|\partial_{x}^{j}A^{-1}(\varphi)\|_{\delta,s+2-j}\le\|\varphi\|_{\delta,s}\;(0\le j\le2),\\
 & \|\partial_{x}^{j}B^{-1}(\varphi)\|_{\delta,s+4-j}\le\|\varphi\|_{\delta,s}\;(0\le j\le4)
\end{align*}
 for $\varphi\in G^{\delta,s}$.
\end{prop}

If $\varphi\in G^{\delta,s}$, then by Propositions \ref{prop:partial_x}
and \ref{prop:smoothing}, we have the following estimates of the
`larger $\delta$, smaller $s$' type: 
\begin{align}
\|\partial_{x}A^{-1}(\varphi)\|_{\delta',s+1} & \le\frac{e^{-1}}{\delta-\delta'}\|\varphi\|_{\delta,s},\;0\le\delta'<\delta\le1,\label{eq:partialx A-1}\\
\|\partial_{x}B^{-1}(\varphi)\|_{\delta',s+3} & \le\frac{e^{-1}}{\delta-\delta'}\|\varphi\|_{\delta,s},\;0\le\delta'<\delta\le1.\label{eq:partialx B-1}
\end{align}
In these estimates, the left-hand sides can be replaced with $\|\partial_{x}A^{-1}(\varphi)\|_{\delta',s+2}$
and $\|\partial_{x}B^{-1}(\varphi)\|_{\delta',s+4}$ respectively,
but \eqref{eq:partialx A-1} and \eqref{eq:partialx B-1} are good
enough.

In later sections, we will use the following estimates repeatedly.
Let $R>0$ and $u_{0}\in G^{\delta,s+1}$ be given. If $\|u_{j}-u_{0}\|_{\delta,s+1}<R$
for $u_{j}\in G^{\delta,s+1}\subset G^{\delta,s}\,(j=1,2)$, then
we have 
\begin{align}
\|u_{j}\|_{\delta,s} & \le\|u_{j}\|_{\delta,s+1}\le\|u_{0}\|_{\delta,s+1}+R,\label{eq:ball1}\\
\|u_{1}+u_{2}\|_{\delta,s} & \le\|u_{1}+u_{2}\|_{\delta,s+1}\le2\left(\|u_{0}\|_{\delta,s+1}+R\right).\label{eq:ball2}
\end{align}

\begin{prop}
\label{prop:KatoMasudaLem2.2}(cf. \cite[Lem 2.2]{Kato-Masuda}) If
$f\in H^{\infty}$ satisfies $\|f\|_{(\sigma,s)}<\infty$ for some
$s\ge0$ and for any $\sigma$ with $e^{2\pi\sigma}<\delta$, then
$f\in A(\delta)$.

Conversely, $f\in A(\delta)$ implies $\|f\|_{(\sigma,s)}<\infty$
for any $\sigma,s$ with $e^{2\pi\sigma}<\delta$ and $s\ge0$.
\end{prop}

\begin{proof}
In the proof of the first part, we may assume $s=0$. Let 
\[
M^{2}=\|f\|_{(\sigma,0)}^{2}=\sum_{j=0}^{\infty}(j!)^{-2}e^{4\pi j\sigma}\|f^{(j)}\|_{0}^{2}.
\]
Then $\|f^{(j)}\|_{0}\le j!e^{-2\pi j\sigma}M$ and 
\[
\|f^{(j)}\|_{0}+\|f^{(j+1)}\|_{0}\le M[j!e^{-2\pi j\sigma}+(j+1)!e^{-2\pi(j+1)\sigma}].
\]
Since there exists $C>0$ such that $\|\varphi\|_{\mathcal{\mathcal{C}}^{0}}\le C(\|\varphi\|_{0}+\|\varphi'\|_{0})$
for any $\varphi\in H^{1}$ by the Sobolev embedding, we have 
\begin{equation}
|f^{(j)}(x)|\le CM\left[j!e^{-2\pi j\sigma}+(j+1)!e^{-2\pi(j+1)\sigma}\right]\label{eq:f^(j)}
\end{equation}
for any $x$. Therefore $f(x+iy)=\sum_{j=0}^{\infty}f^{(j)}(x)(iy)^{j}/j!$
is holomorphic in $|y|<e^{2\pi\sigma}$ for any $\sigma$ with $e^{2\pi\sigma}<\delta$.
Hence $f\in A(\delta)$.

Next, we show the second part. Assume $0<\delta'<\delta$. It is enough
to prove $\|f\|_{(\sigma,s)}^{2}<\infty$ for any $\sigma,s$ with
$e^{2\pi\sigma}<\delta'$. Let $S=\sup_{S(\delta')}|f|<\infty$. Set
$L_{\pm}=\left\{ t\pm iy;\,0\le t\le1\right\} $. If $0<|y|\le\delta'$,
the periodicity of $f$ and Goursat's formula yield
\[
f^{(j)}(x)=\frac{j!}{2\pi i}\left(\int_{L_{-}}-\int_{L_{+}}\right)\frac{f(\zeta)}{(x-\zeta)^{j+1}}\,d\zeta,\;0\le x\le1.
\]
 We have $|f^{(j)}(x)|\le\pi^{-1}j!S/y^{j+1}$. Then $y^{2(j+1)}\|f^{(j)}\|_{0}^{2}\le\pi^{-2}j!^{2}S^{2}$
in $|y|\le\delta'$. Integrating in $y\in[0,\delta']$, we get 
\[
(\delta')^{2j+2}\|f^{(j)}\|_{0}^{2}\le\pi^{-2}(2j+3)j!^{2}S^{2}.
\]
It follows that
\begin{equation}
\|f\|_{(\sigma,0)}^{2}=\sum_{j=0}^{\infty}(j!)^{-2}e^{4\pi j\sigma}\|f^{(j)}\|_{0}^{2}\le\sum_{j=0}^{\infty}\pi^{-2}(2j+3)(\delta')^{-(2j+2)}e^{4\pi j\sigma}S^{2}<\infty.\label{eq:f2sigma0}
\end{equation}
For $s>0$, we can prove that
\begin{equation}
\|f\|_{(\sigma,s)}\le\text{const.}\|f\|_{(\sigma',0)}<\infty,\label{eq:fsigmas}
\end{equation}
where $e^{2\pi\sigma}<e^{2\pi\sigma'}<\delta'$. To see this, we may
assume that $s$ is a positive integer in view of $\|f\|_{(\sigma,s_{1})}\le\|f\|_{(\sigma,s_{2})}\,(s_{1}\le s_{2})$.
For simplicity, we explain the case of $s=1$ only. (The general case
follows the same line of proof, the only additional tool being the
binomial expansion of powers of the Japanese bracket.) We have
\begin{align*}
\|f\|_{(\sigma,1)}^{2} & =\sum_{j=0}^{\infty}(j!)^{-2}e^{4\pi j\sigma}\|f^{(j)}\|_{1}^{2}=\sum_{j=0}^{\infty}(j!)^{-2}e^{4\pi j\sigma}\sum_{k\in\mathbb{Z}}(1+k^{2})(2\pi k)^{2j}|\hat{f}(k)|^{2}\\
 & =I_{0}+I_{1},
\end{align*}
where $I_{p}=\sum_{j=0}^{\infty}(j!)^{-2}e^{4\pi j\sigma}\sum_{k\in\mathbb{Z}}k^{2p}(2\pi k)^{2j}|\hat{f}(k)|^{2}\,(p=0,1)$.
Obviously $I_{0}=\|f\|_{(\sigma,0)}^{2}\le\|f\|_{(\sigma',0)}^{2}$.
On the other hand, setting $\ell=j+1$, we get
\begin{align*}
I_{1} & \le\sum_{\ell=1}^{\infty}\frac{\ell^{2}e^{-4\pi\sigma}e^{4\pi\ell(\sigma-\sigma')}}{\ell!^{2}}e^{4\pi\ell\sigma'}\sum_{k\in\mathbb{Z}}(2\pi k)^{2l}|\hat{f}(k)|^{2}\\
 & \le\sum_{\ell=0}^{\infty}\frac{\textrm{const.}}{\ell!^{2}}e^{4\pi\ell\sigma'}\sum_{k\in\mathbb{Z}}(2\pi k)^{2l}|\hat{f}(k)|^{2}=\sum_{\ell=0}^{\infty}\frac{\textrm{const.}}{\ell!^{2}}e^{4\pi\ell\sigma'}\|f^{(\ell)}\|_{0}^{2}=\textrm{const.}\|f\|_{(\sigma',0)}^{2}.
\end{align*}
The proof of the second part is over.
\end{proof}
\begin{lem}
If $\sigma<\sigma'$, we have 
\[
\|f\|_{(\sigma,s)}\le\mathrm{const.}\|f\|_{(\sigma',s')}
\]
for any $s,s'\ge0$, where the constant depends on $s,s',\sigma,\sigma'$.
\end{lem}

\begin{proof}
This estimate follows from \eqref{eq:fsigmas} immediately.
\end{proof}
\begin{cor}
\label{cor:Frechet}The following four families of norms on $A(\delta)$
determine the same topology as a Fr\'echet space.

(i) $\sup_{z\in S(\delta')}|f(z)|\,(0<\delta'<\delta),\qquad$ (ii)
$\|\cdot\|_{(\sigma,s)}\,(e^{2\pi\sigma}<\delta,s\ge0),$

(iii) $\|\cdot\|_{(\sigma,2)}\,(e^{2\pi\sigma}<\delta),\qquad$(iv)
$\|\cdot\|_{(\sigma,0)}\,(e^{2\pi\sigma}<\delta)$.\\
With this topology, $A(\delta)$ is continuously embedded in $H^{\infty}(S^{1})=\cap_{s\ge0}H^{s}(S^{1})$.
\end{cor}

\begin{proof}
It is trivial that (ii) is stronger (not weaker) than (iv) and that
(iii) is between (ii) and (iv). On the other hand, \eqref{eq:fsigmas}
implies (iv) is stronger than (ii).

The estimate \eqref{eq:f^(j)} and the Taylor expansion imply that
(iv) is stronger than (i). On the other hand, \eqref{eq:f2sigma0}
implies (i) is stronger than (iv). 
\end{proof}
\begin{prop}
\label{prop:KatoMasudaLem2.4}(cf. \cite[Lemma 2.4]{Kato-Masuda})
Let $f_{n}\in A(\delta)\,(n=0,1,2,\dots)$ be a sequence with $\|f_{n}\|_{(\sigma,s)}$
bounded, where $e^{2\pi\sigma}<\delta$, and assume $f_{\infty}\in A(\delta)$.
If $f_{n}\to f_{\infty}$ in $H^{\infty}$ as $n\to\infty$, then
$\|f_{n}\|_{(\sigma',s)}\to f_{\infty}$ for each $\sigma'<\sigma$.
\end{prop}

\begin{proof}
We may assume $f_{\infty}=0$. Let $M^{2}=\sup_{n}\|f_{n}\|_{(\sigma,s)}^{2}$.
Then $j!^{-2}e^{4\pi j\sigma'}\|f_{n}^{(j)}\|_{s}^{2}\le Me^{4\pi(\sigma'-\sigma)j}$.
Since $\lim_{n\to\infty}\|f_{n}^{(j)}\|_{s}=0$ for each $j$ and
$\sum_{j\ge0}M^{2}e^{4\pi(\sigma'-\sigma)j}<\infty$, we can apply
(the sum version of ) Lebesgue's dominated convergence theorem to
$\|f_{n}\|_{(\sigma',s)}^{2}=\sum_{j\ge0}j!^{-2}e^{4\pi j\sigma'}\|f_{n}^{(j)}\|_{s}^{2}\,(n=0,1,2,\dots)$.
\end{proof}

\section{Local-in-time solutions\label{sec:Local-in-time solutions}}

\subsection{Autonomous Ovsyannikov theorem}

We recall some basic facts about the autonomous Ovsyannikov theorem.
Among many versions, we adopt the one in \cite{BHP,BHPpower}. Let
$\left\{ X_{\delta},\|\cdot\|_{\delta}\right\} _{0<\delta\le1}$ be
a (decreasing) scale of Banach spaces, i.e. each $X_{\delta}$ is
a Banach space and $X_{\delta}\subset X_{\delta'},\|\cdot\|_{\delta'}\le\|\cdot\|_{\delta}$
for any $0<\delta'<\delta\le1$. (For $s$ fixed, $\left\{ G^{\delta,s},\|\cdot\|_{\delta,s}\right\} _{0<\delta\le1}$
is a scale of Banach spaces.) Assume that $F\colon X_{\delta}\to X_{\delta'}$
is a mapping satisfying the following conditions.

(a) For any $u_{0}\in X_{1}$ and $R>0$, there exist $L=L(u_{0},R)>0,M=M(u_{0},R)>0$
such that we have 
\begin{align}
\|F(u_{0})\|_{\delta} & \le\frac{M}{1-\delta}\label{eq:Ov Fu0}
\end{align}
if $0<\delta<1$ and 
\begin{equation}
\|F(u)-F(v)\|_{\delta'}\le\frac{L}{\delta-\delta'}\|u-v\|_{\delta}\label{eq:Ov Lipschitz}
\end{equation}
if $0<\delta'<\delta\le1$ and $u,v\in X_{\delta}$ satisfies $\|u-u_{0}\|_{\delta}<R,\|v-u_{0}\|_{\delta}<R$.

(b) If $u(t)$ is holomorphic on the disk $D(0,a(1-\delta))=\left\{ t\in\mathbb{C}\colon|t|<a(1-\delta)\right\} $
with values in $X_{\delta}$ for $a>0,0<\delta<1$ satisfying $\sup_{|t|<a(1-\delta)}\|u(t)-u_{0}\|_{\delta}<R$,
then the composite function $F(u(t))$ is a holomorphic function on
$D(0,a(1-\delta))$ with values in $X_{\delta'}$ for any $0<\delta'<\delta$.

The autonomous Ovsyannikov theorem below is our main tool. For the
proof, see \cite{BHP}.
\begin{thm}
\label{thm:Ov}Assume that the mapping $F$ satisfies the conditions
(a) and (b). For any $u_{0}\in X_{1}$ and $R>0$, set 
\begin{equation}
T=\frac{R}{16LR+8M}.\label{eq:T}
\end{equation}
Then, for any $\delta\in]0,1[$, the Cauchy problem 
\begin{equation}
\frac{du}{dt}=F(u),\;u(0)=u_{0}\label{eq:Ov CP}
\end{equation}
has a unique holomorphic solution $u(t)$ in the disk $D(0,T(1-\delta))$
with values in $X_{\delta}$ satisfying
\[
\sup_{|t|<T(1-\delta)}\|u(t)-u_{0}\|_{\delta}<R.
\]
\end{thm}

\subsection{$\mu$CH and $\mu$DP equations}

First we consider the analytic Cauchy problem for the $\mu$CH equation
 \eqref{eq:muCH2}, namely. 
\begin{equation}
\begin{cases}
u_{t}+uu_{x}+\partial_{x}A^{-1}\left[2\mu(u)u+\frac{1}{2}u_{x}^{2}\right]=0,\\
u(0,x)=u_{0}(x).
\end{cases}\label{eq:muCH IVP}
\end{equation}

\begin{thm}
\label{thm:muCH IVP}Let $s>1/2$. If $u_{0}\in G^{1,s+1}$, then
there exists a positive time $T=T(u_{0},s)$ such that for every $\delta\in]0,1[$,
the Cauchy problem  \eqref{eq:muCH IVP} has a unique solution which
is a holomorphic function valued in $G^{\delta,s+1}$ in the disk
$D(0,T(1-\delta))$. Furthermore, the analytic lifespan $T$ satisfies
\[
T=\frac{\mathrm{const.}}{\|u_{0}\|_{1,s+1}}.
\]
\end{thm}

\begin{proof}
Assume $\|u-u_{0}\|_{\delta,s+1}<R,\|v-u_{0}\|_{\delta,s+1}<R$. By
Proposition \ref{prop:multiplication} (i), the first inequality in
Proposition \ref{prop:partial_x} and \eqref{eq:ball2}, 
\begin{align}
\|(u^{2})_{x}-(v^{2})_{x}\|_{\delta',s+1} & \le\frac{e^{-1}}{\delta-\delta'}\|u^{2}-v^{2}\|_{\delta,s+1}\label{eq:u^2}\\
 & \le\frac{e^{-1}c_{s+1}}{\delta-\delta'}\|u+v\|_{\delta,s+1}\|u-v\|_{\delta,s+1}\nonumber \\
 & \leq\frac{2e^{-1}c_{s+1}}{\delta-\delta'}\left(\|u_{0}\|_{\delta,s+1}+R\right)\|u-v\|_{\delta,s+1}.\nonumber 
\end{align}

On the other hand, since we have $\mu(u)u-\mu(v)v=\mu(u-v)u+\mu(v)(u-v)$,
Proposition \ref{prop:smoothing} and \eqref{eq:ball1} imply
\begin{align*}
\|\mu(u)u-\mu(v)v\|_{\delta,s} & \le|\mu(u-v)|\|u\|_{\delta,s}+|\mu(v)|\|u-v\|_{\delta,s}\\
 & \le(\|u\|_{\delta,s}+\|v\|_{\delta,s})\|u-v\|_{\delta,s}\\
 & \le2\left(\|u_{0}\|_{\delta,s+1}+R\right)\|u-v\|_{\delta,s+1}.
\end{align*}
Therefore \eqref{eq:partialx A-1} yields
\begin{align}
\left\Vert \partial_{x}A^{-1}\left[\mu(u)u-\mu(v)v\right]\right\Vert _{\delta',s+1} & \le\frac{e^{-1}}{\delta-\delta'}\left\Vert \mu(u)u-\mu(v)v\right\Vert _{\delta,s}\label{eq:partial A^-1 mu(u)u}\\
 & \le\frac{2e^{-1}}{\delta-\delta'}\left(\|u_{0}\|_{\delta,s+1}+R\right)\|u-v\|_{\delta,s+1}.\nonumber 
\end{align}
Next, by Proposition \ref{prop:multiplication} (i) and the second
inequality in Proposition \ref{prop:partial_x}, 
\begin{align*}
\|u_{x}^{2}-v_{x}^{2}\|_{\delta,s} & \le c_{s}\|u_{x}+v_{x}\|_{\delta,s}\|u_{x}-v_{x}\|_{\delta,s}\le4\pi^{2}c_{s}\|u+v\|_{\delta,s+1}\|u-v\|_{\delta,s+1}\\
 & \le8\pi^{2}c_{s}\left(\|u_{0}\|_{\delta,s+1}+R\right)\|u-v\|_{\delta,s+1}.
\end{align*}
Hence \eqref{eq:partialx A-1} gives
\begin{align}
\left\Vert \partial_{x}A^{-1}\left(u_{x}^{2}-v_{x}^{2}\right)\right\Vert _{\delta',s+1} & \le\frac{e^{-1}}{\delta-\delta'}\left\Vert u_{x}^{2}-v_{x}^{2}\right\Vert _{\delta,s}\label{eq:partial A^-1 u_x^2}\\
 & \le\frac{8\pi^{2}e^{-1}c_{s}}{\delta-\delta'}\left(\|u_{0}\|_{\delta,s+1}+R\right)\|u-v\|_{\delta,s+1}.\nonumber 
\end{align}

Now set 
\begin{align}
F_{\mu}(u) & =-uu_{x}-\partial_{x}A^{-1}\left[2\mu(u)u+\frac{1}{2}u_{x}^{2}\right]\label{eq:F muCH}\\
 & =-\frac{1}{2}(u^{2})_{x}-\partial_{x}A^{-1}\left[2\mu(u)u+\frac{1}{2}u_{x}^{2}\right].\nonumber 
\end{align}
Then  \eqref{eq:u^2},  \eqref{eq:partial A^-1 mu(u)u} and  \eqref{eq:partial A^-1 u_x^2}
give the Lipschitz continuity of $F_{\mu}$: 
\begin{equation}
\|F_{\mu}(u)-F_{\mu}(v)\|_{\delta',s+1}\le\frac{L}{\delta-\delta'}\|u-v\|_{\delta,s+1},\label{eq:Lipschitz}
\end{equation}
where $L=C\left(\|u_{0}\|_{1,s+1}+R\right),C=e^{-1}(c_{s+1}+4+4\pi^{2}c_{s})$.

Next we will derive an estimate of $\|F_{\mu}(u_{0})\|_{\delta,s+1}$.
Since
\begin{align*}
 & \|(u_{0}^{2})_{x}\|_{\delta,s+1}\le\frac{e^{-1}}{1-\delta}\|u_{0}^{2}\|_{1,s+1}\le\frac{e^{-1}c_{s+1}}{1-\delta}\|u_{0}\|_{1,s+1}^{2},\\
 & \|\partial_{x}A^{-1}[\mu(u_{0})u_{0}]\|_{\delta,s+1}\le\frac{e^{-1}}{1-\delta}\|\mu(u_{0})u_{0}\|_{1,s}\le\frac{e^{-1}}{1-\delta}\|u_{0}\|_{1,s+1}^{2},\\
 & \|\partial_{x}A^{-1}(\partial_{x}u_{0})^{2}\|_{\delta,s+1}\le\frac{e^{-1}}{1-\delta}\|(\partial_{x}u_{0})^{2}\|_{1,s}\le\frac{e^{-1}c_{s}}{1-\delta}\|\partial_{x}u_{0}\|_{1,s}^{2}\\
 & \le\frac{(2\pi)^{2}e^{-1}c_{s}}{1-\delta}\|u_{0}\|_{1,s+1}^{2},
\end{align*}
we have
\[
\|F_{\mu}(u_{0})\|_{\delta,s+1}\le\frac{M}{1-\delta},\;M=\frac{C}{2}\|u_{0}\|_{1,s+1}^{2}.
\]
We set 
\[
T=\frac{R}{16LR+8M}=\frac{R}{4C\left[4R\left(\|u_{0}\|_{1,s+1}+R\right)+\|u_{0}\|_{1,s+1}^{2}\right]}.
\]
Because of Theorem \ref{thm:Ov}, there exists a unique solution $u=u(t)$
to  \eqref{eq:muCH IVP} which is a holomorphic mapping from $D(0,T(1-\delta))$
to $G^{\delta,s+1}$ and 
\[
\sup_{|t|<T(1-\delta)}\|u(t)-u_{0}\|_{\delta,s+1}<R.
\]
If we set $R=\|u_{0}\|_{1,s+1},$ we have 
\[
T=\frac{1}{36e^{-1}(c_{s+1}+4+4\pi^{2}c_{s})\|u_{0}\|_{1,s+1}}.
\]
\end{proof}
In Theorem \ref{thm:muCH IVP}, we assumed the initial value $u_{0}$
was in $G^{1,s+1}$. We can relax this assumption as in the following
theorem. 
\begin{thm}
\label{thm:muCH IVP2}If $u_{0}$ is a real-analytic function on $S^{1}$,
then the Cauchy problem  \eqref{eq:muCH IVP} has a holomorphic solution
near $t=0$. More precisely, we have the following:

(i) There exists $\Delta>0$ such that $u_{0}\in G^{\Delta,s+1}\subset A(\Delta)$
for any $s$.

(ii) If $s>1/2$, there exists a positive time $T_{\Delta}=T(u_{0},s,\Delta)$
such that for every $d\in]0,1[$, the Cauchy problem  \eqref{eq:muCH IVP}
has a unique solution which is a holomorphic function valued in $G^{\Delta d,s+1}$
in the disk $D(0,T_{\Delta}(1-d))$. Furthermore, the analytic lifespan
$T_{\Delta}$ satisfies 
\[
T_{\Delta}=\frac{\mathrm{const.}}{\|u_{0}\|_{\Delta,s+1}}
\]
when $\Delta$ is fixed.
\end{thm}

\begin{proof}
The first statement is nothing but Proposition \ref{prop:analyticity}. 

Set $X_{d}=G^{\Delta d,s+1},\|\cdot\|_{d,s+1}^{(\Delta)}=\|\cdot\|_{\Delta d,s+1}$.
Then $\left\{ X_{d},\|\cdot\|_{d,s+1}^{(\Delta)}\right\} _{0<d\le1}$
is a (decreasing) scale of of Banach spaces and $u_{0}\in X_{1}$.

Assume $\|u-u_{0}\|_{d,s+1}^{(\Delta)}<R,\|v-u_{0}\|_{d,s+1}^{(\Delta)}<R$
and $0<d'<d\le1$. Then  \eqref{eq:u^2},  \eqref{eq:partial A^-1 mu(u)u}
and \eqref{eq:partial A^-1 u_x^2} give ($\delta=\Delta d,\delta'=\Delta d'$)
the following counterpart of \eqref{eq:Lipschitz}:
\begin{equation}
\|F_{\mu}(u)-F_{\mu}(v)\|_{d',s+1}^{(\Delta)}\le\frac{L_{\Delta}}{d-d'}\|u-v\|_{d,s+1}^{(\Delta)},\label{eq:Lipschitz-1}
\end{equation}
where $L_{\Delta}=\Delta^{-1}C\left(\|u_{0}\|_{1,s+1}^{(\Delta)}+R\right)=C\left(\|u_{0}\|_{\Delta,s+1}+R\right),C=e^{-1}(c_{s+1}+4+4\pi^{2}c_{s})$.
Simpler estimates give 
\[
\|F_{\mu}(u_{0})\|_{d,s+1}^{(\Delta)}\le\frac{M_{\Delta}}{1-d},\;M_{\Delta}=\frac{C}{2\Delta}\left(\|u_{0}\|_{1,s+1}^{(\Delta)}\right)^{2}.
\]
 We set 
\[
T_{\Delta}=\frac{R}{16L_{\Delta}R+8M_{\Delta}}=\frac{R\Delta}{4C\left[4R\left(\|u_{0}\|_{\Delta,s+1}+R\right)+\|u_{0}\|_{\Delta,s+1}^{2}\right]}.
\]
Because of Theorem \ref{thm:Ov}, there exists a unique solution $u=u(t)$
to  \eqref{eq:muCH IVP} which is a holomorphic mapping from $D(0,T(1-\delta))$
to $X_{d}=G^{\Delta d,s+1}$ and 
\[
\sup_{|t|<T(1-d)}\|u(t)-u_{0}\|_{\Delta d,s+1}<R.
\]
If we set $R=\|u_{0}\|_{1,s+1}^{(\Delta)}=\|u_{0}\|_{\Delta,s+1},$
we have 
\[
T_{\Delta}=\frac{\Delta}{36e^{-1}(c_{s+1}+4+4\pi^{2}c_{s})\|u_{0}\|_{\Delta,s+1}}.
\]
\end{proof}
We can study the following Cauchy problem for the $\mu$DP equation
\eqref{eq:muDP} by using the same estimates \eqref{eq:u^2} and \eqref{eq:partial A^-1 mu(u)u}.
\begin{equation}
\begin{cases}
u_{t}+uu_{x}+\partial_{x}A^{-1}\left[3\mu(u)u\right]=0,\\
u(0,x)=u_{0}(x).
\end{cases}\label{eq:muDP IVP}
\end{equation}
\begin{comment}
(i) Let $s>1/2$. If $u_{0}\in G^{1,s+1}$, then there exists a positive
time $T=T(u_{0},s)$ such that for every $\delta\in(0,1)$, the Cauchy
problem \eqref{eq:muDP IVP} has a unique solution which is a holomorphic
function valued in $G^{\delta,s+1}$ in the disk $D(0,T(1-\delta))$.
Furthermore, the analytic lifespan $T$ satisfies the estimate 
\[
T\approx\frac{1}{\|u_{0}\|_{1,s+1}}.
\]
\end{comment}

\begin{thm}
\label{thm:muDP IVP} If $u_{0}$ is a real-analytic function on $S^{1}$,
then the Cauchy problem \eqref{eq:muDP IVP} has a holomorphic solution
near $t=0$. More precisely, we have the following: \\
(i) There exists $\Delta>0$ such that $u_{0}\in G^{\Delta,s+1}\subset A(\Delta)$
for any $s$. \\
(ii) If $s>1/2$, there exists a positive time $T_{\Delta}=T(u_{0},s,\Delta)$
such that for every $d\in]0,1[$, the Cauchy problem  \eqref{eq:muDP IVP}
has a unique solution which is a holomorphic function valued in $G^{\Delta d,s+1}$
in the disk $D(0,T_{\Delta}(1-d))$. Furthermore, the analytic lifespan
$T_{\Delta}$ satisfies 
\[
T_{\Delta}=\frac{\mathrm{const.}}{\|u_{0}\|_{\Delta,s+1}}
\]
when $\Delta$ is fixed.
\end{thm}

\subsection{Higher-order $\mu$CH equation}

We consider the analytic Cauchy problem for the higher-order $\mu$CH
equation \eqref{eq:higher}, namely. 
\begin{equation}
\begin{cases}
u_{t}+uu_{x}+\partial_{x}B^{-1}\left[2\mu(u)u+\frac{1}{2}u_{x}^{2}-3u_{x}u_{xxx}-\frac{7}{2}u_{xx}^{2}\right]=0,\\
u(0,x)=u_{0}(x).
\end{cases}\label{eq:higher muCH IVP}
\end{equation}

\begin{thm}
\label{thm:higher muCH IVP}Let $s>1/2$. If $u_{0}\in G^{1,s+3}$,
then there exists a positive time $T=T(u_{0},s)$ such that for every
$\delta\in]0,1[$, the Cauchy problem \eqref{eq:higher muCH IVP}
has a unique solution which is a holomorphic function valued in $G^{\delta,s+3}$
in the disk $D(0,T(1-\delta))$. Furthermore, the analytic lifespan
$T$ satisfies 
\[
T=\frac{\mathrm{const.}}{\|u_{0}\|_{1,s+3}}.
\]
\end{thm}

\begin{proof}
Assume $u_{0}\in G^{\delta,s+3}$ and $\|u-u_{0}\|_{\delta,s+3}<R,\|v-u_{0}\|_{\delta,s+3}<R$.
We follow the proofs of \eqref{eq:u^2},\eqref{eq:partial A^-1 mu(u)u}
and \eqref{eq:partial A^-1 u_x^2} with \eqref{eq:partialx B-1} instead
of \eqref{eq:partialx A-1} to obtain
\begin{align}
 & \|(u^{2})_{x}-(v^{2})_{x}\|_{\delta',s+3}\le\frac{2e^{-1}c_{s+3}}{\delta-\delta'}(\|u_{0}\|_{\delta,s+3}+R)\|u-v\|_{\delta,s+3},\label{eq:higher (u^2)_x}\\
 & \|\partial_{x}B^{-1}[\mu(u)u-\mu(v)v]\|_{\delta',s+3}\le\frac{2e^{-1}}{\delta-\delta'}(\|u_{0}\|_{\delta,s+3}+R)\|u-v\|_{\delta,s+3},\label{eq:partial B^-1 mu(u)u}\\
 & \|\partial_{x}B^{-1}[u_{x}^{2}-v_{x}^{2}]\|_{\delta',s+3}\le\frac{8\pi^{2}e^{-1}c_{s}}{\delta-\delta'}(\|u_{0}\|_{\delta,s+3}+R)\|u-v\|_{\delta,s+3}.\label{eq:partial B^-1 u_x^2}
\end{align}

Next we study the difference associated with $\partial_{x}B^{-1}(u_{x}u_{xxx})$.
Since $u_{x}u_{xxx}-v_{x}v_{xxx}=(u_{x}-v_{x})u_{xxx}+v_{x}(u_{xxx}-v_{xxx})$,
we have 
\begin{align*}
 & \|u_{x}u_{xxx}-v_{x}v_{xxx}\|_{\delta,s}\\
 & \le c_{s}\left(\|u_{x}-v_{x}\|_{\delta,s}\|u_{xxx}\|_{\delta,s}+\|v_{x}\|_{\delta,s}\|u_{xxx}-v_{xxx}\|_{\delta,s}\right).
\end{align*}
Since 
\begin{align*}
\|u_{x}-v_{x}\|_{\delta,s} & \le2\pi\|u-v\|_{\delta,s+3},\\
\|u_{xxx}\|_{\delta,s} & \le(2\pi)^{3}\|u\|_{\delta,s+3}\le(2\pi)^{3}(\|u_{0}\|_{\delta,s+3}+R),\\
\|v_{x}\|_{\delta,s} & \le2\pi(\|u_{0}\|_{\delta,s+3}+R),\\
\|u_{xxx}-v_{xxx}\|_{\delta,s} & \le(2\pi)^{3}\|u-v\|_{\delta,s+3},
\end{align*}
we have 
\[
\|u_{x}u_{xxx}-v_{x}v_{xxx}\|_{\delta,s}\le2^{5}\pi{}^{4}c_{s}(\|u_{0}\|_{\delta,s+3}+R)\|u-v\|_{\delta,s+3}.
\]
Therefore by \eqref{eq:partialx B-1} 
\begin{align}
\|\partial_{x}B^{-1}[u_{x}u_{xxx}-v_{x}v_{xxx}]\|_{\delta',s+3} & \le\frac{2^{5}\pi{}^{4}e^{-1}c_{s}}{\delta-\delta'}(\|u_{0}\|_{\delta,s+3}+R)\|u-v\|_{\delta,s+3}.\label{eq:partial B^-1 u_x u_xxx}
\end{align}
Next we study the difference associated with $\partial_{x}B^{-1}(u_{xx}^{2})$.
We have $u_{xx}^{2}-v_{xx}^{2}=(u_{xx}+v_{xx})(u_{xx}-v_{xx})$ and
\[
\|u_{xx}^{2}-v_{xx}^{2}\|_{\delta,s}\le c_{s}\|u_{xx}+v_{xx}\|_{\delta,s}\|u_{xx}-v_{xx}\|_{\delta,s}.
\]
Since
\begin{align*}
\|u_{xx}+v_{xx}\|_{\delta,s} & \le(2\pi)^{2}\|u+v\|_{\delta,s+2}\le(2\pi)^{2}\|u+v\|_{\delta,s+3}\\
 & \le2^{3}\pi{}^{2}(\|u_{0}\|_{\delta,s+3}+R),\\
\|u_{xx}-v_{xx}\|_{\delta,s} & \le(2\pi)^{2}\|u-v\|_{\delta,s+2}\le2^{2}\pi{}^{2}\|u-v\|_{\delta,s+3},
\end{align*}
we have 
\[
\|u_{xx}^{2}-v_{xx}^{2}\|_{\delta,s}\le2^{5}\pi{}^{4}c_{s}(\|u_{0}\|_{\delta,s+3}+R)\|u-v\|_{\delta,s+3},
\]
and 
\begin{align}
\|\partial_{x}B^{-1}[u_{xx}^{2}-v_{xx}^{2}]\|_{\delta',s+3} & \le\frac{e^{-1}}{\delta-\delta'}\|B^{-1}[u_{xx}^{2}-v_{xx}^{2}]\|_{\delta,s}\label{eq:partial B^-1 u_xx^2}\\
 & \le\frac{2^{5}\pi{}^{4}e^{-1}c_{s}}{\delta-\delta'}(\|u_{0}\|_{\delta,s+3}+R)\|u-v\|_{\delta,s+3}.\nonumber 
\end{align}
Now we set 
\begin{equation}
G(u)=-\frac{1}{2}(u^{2})_{x}-\partial_{x}B^{-1}\left[2\mu(u)u+\frac{1}{2}u_{x}^{2}-3u_{x}u_{xxx}-\frac{7}{2}u_{xx}^{2}\right].\label{eq:G}
\end{equation}
Then by \eqref{eq:higher (u^2)_x}, \eqref{eq:partial B^-1 mu(u)u},
\eqref{eq:partial B^-1 u_x^2}, \eqref{eq:partial B^-1 u_x u_xxx}
and \eqref{eq:partial B^-1 u_xx^2}, we obtain
\begin{align}
 & \|G(u)-G(v)\|_{\delta',s+3}\le\frac{L_{G}}{\delta-\delta'}\|u-v\|_{\delta,s+3},\label{eq:Gu-Gv}\\
 & L_{G}=e^{-1}\left[c_{s+3}+4+(4\pi^{2}+208\pi{}^{4})c_{s}\right](\|u_{0}\|_{\delta,s+3}+R).
\end{align}

We need an estimate of $G(u_{0})$. By using
\begin{align*}
 & \|(u_{0}^{2})_{x}\|_{\delta,s+3}\le\frac{e^{-1}}{1-\delta}\|u_{0}^{2}\|_{1,s+3}\le\frac{e^{-1}c_{s+3}}{1-\delta}\|u_{0}\|_{1,s+3}^{2},\\
 & \|\partial_{x}B^{-1}[\mu(u_{0})u_{0}]\|_{\delta,s+3}\le\frac{e^{-1}}{1-\delta}\|u_{0}\|_{1,s+3}^{2},\\
 & \|\partial_{x}B^{-1}[(\partial_{x}u_{0})^{2}]\|_{\delta,s+3}\le\frac{e^{-1}c_{s}}{1-\delta}\|\partial_{x}u_{0}\|_{1,s}^{2}\le\frac{e^{-1}(2\pi)^{2}c_{s}}{1-\delta}\|u_{0}\|_{1,s+3}^{2},\\
 & \|\partial_{x}B^{-1}[(u_{0})_{x}(u_{0})_{xxx}]\|_{\delta,s+3}\le\frac{e^{-1}}{1-\delta}c_{s}\|(u_{0})_{x}\|_{1,s}\|(u_{0})_{xxx}\|_{1,s}\\
 & \le\frac{e^{-1}(2\pi)^{4}}{1-\delta}c_{s}\|u_{0}\|_{1,s+3}^{2},\\
 & \|\partial_{x}B^{-1}[(u_{0})_{xx}^{2}]\|_{\delta,s+3}\le\frac{e^{-1}}{1-\delta}c_{s}\|(u_{0})_{xx}\|_{1,s}^{2}\le\frac{e^{-1}(2\pi)^{4}}{1-\delta}c_{s}\|u_{0}\|_{1,s+3}^{2},
\end{align*}
we obtain
\begin{align}
 & \|G(u_{0})\|_{\delta',s+3}\le\frac{M_{G}}{1-\delta},\,\label{eq:Gu_0}\\
 & M_{G}=e^{-1}\left[\frac{1}{2}c_{s+3}+2+(2\pi^{2}+104\pi^{4})c_{s}\right]\|u_{0}\|_{1,s+3}^{2}.
\end{align}

We set 
\[
T=\frac{R}{16L_{G}R+8M_{G}}.
\]
Because of Theorem \ref{thm:Ov}, there exists a unique solution $u=u(t)$
to \eqref{eq:higher muCH IVP} which is a holomorphic mapping from
$D(0,T(1-\delta))$ to $G^{\delta,s+3}$ and 
\[
\sup_{|t|<T(1-\delta)}\|u(t)-u_{0}\|_{\delta,s+3}<R.
\]
If we set $R=\|u_{0}\|_{1,s+3}$, we have 
\[
T=\frac{1}{160e^{-1}\left[c_{s+3}/4+1+(\pi^{2}+52\pi^{4})c_{s}\right]\|u_{0}\|_{1,s+3}}.
\]
\end{proof}
Similarly, we have the following theorem about the other higher order
$\mu$CH \eqref{eq:higher-1}.
\begin{thm}
Let $s>1/2$. If $u_{0}\in G^{1,s+3}$, then there exists a positive
time $T=T(u_{0},s)$ such that for every $\delta\in(0,1)$, the Cauchy
problem 
\begin{equation}
\begin{cases}
 & u_{t}+uu_{x}+\partial_{x}A^{-2}\left[2\mu(u)u+\frac{1}{2}u_{x}^{2}-3u_{x}u_{xxx}-\frac{7}{2}u_{xx}^{2}\right]=0,\\
 & u(0,x)=u_{0}(x)
\end{cases}\label{eq:the other higher CP}
\end{equation}
has a unique solution which is a holomorphic function valued in $G^{\delta,s+3}$
in the disk $D(0,T(1-\delta))$. Furthermore, the analytic lifespan
$T$ satisfies 
\[
T=\frac{\mathrm{const.}}{\|u_{0}\|_{1,s+3}}.
\]
\end{thm}

\begin{proof}
The proof is almost the same as for Theorem \ref{thm:higher muCH IVP}.
Indeed, $A^{-2}$ has the same properties as those of $B^{-1}$.
\end{proof}
Global-in-time analytic solutions of \eqref{eq:muCH IVP}, \eqref{eq:muDP IVP},
\eqref{eq:higher muCH IVP} and \eqref{eq:higher-1} will be studied
in the following sections. The proofs will rely on known results about
the global existence in Sobolev spaces. On the other hand, that kind
of existence for the non-quasilinear equation \eqref{eq:modifiedmuCH2}
is unknown. Therefore, the argument given below is not valid for it.
The local theory of \eqref{eq:modifiedmuCH2} will be given in Appendix.

\section{Global-in-time solutions\label{sec:Global-in-time-solutions}}

\subsection{Statement of the main results}

We recall known results about global-in-time solutions to $\mu$CH
and $\mu$DP in Sobolev spaces. First, as for $\mu$CH, we have
\begin{thm}
(\cite[Theorem 5.1, 5.5]{KLM})\label{thm:KLM} Let $s>5/2$. Assume
that $u_{0}\in H^{s}(S^{1})$ has non-zero mean and satisfies the
condition
\[
(\mu-\partial_{x}^{2})u_{0}\ge0\;(\mathrm{or}\text{\;\ensuremath{\le0)}}.
\]
Then \eqref{eq:muCH IVP} has a unique global-in-time solution in
$\mathcal{C}(\mathbb{R},H^{s}(S^{1}))\cap\mathcal{C}^{1}(\mathbb{R},H^{s-1}(S^{1}))$.
\\
Moreover, local well-posedness (in particular, uniqueness) holds. 
\end{thm}

The quantity $(\mu-\partial_{x}^{2})u_{0}$ is the $\mu$-version
of the McKean quantity $(1-\partial_{x}^{2})u_{0}$ (\cite{McKean}).
There is an analogous result about $\mu$DP.
\begin{thm}
(\cite[Theorem 5,1, 5.4]{LMT})\label{thm:LMT} Let $s>3$. Assume
that $u_{0}\in H^{s}(S^{1})$ has non-zero mean and satisfies the
condition
\[
(\mu-\partial_{x}^{2})u_{0}\ge0\;(\mathrm{or}\text{\;\ensuremath{\le0)}}.
\]
Then \eqref{eq:muDP IVP} has a unique global-in-time solution in
$\mathcal{C}(\mathbb{R},H^{s}(S^{1}))\cap\mathcal{C}^{1}(\mathbb{R},H^{s-1}(S^{1}))$.
\\
Moreover, local well-posedness holds.
\end{thm}

We will discuss global-in-time analytic solutions. These solutions
are analytic in both the time and space variables. Notice that in
the KdV and other cases treated in \cite{Kato-Masuda}, solutions
are analytic in the space variable only. This is due to the absence
of Cauchy-Kowalevsky type theorems. Our main result about the $\mu$-Camassa-Holm
equation is the following:
\begin{thm}
\label{thm:muCH main}Assume that a real-analytic function $u_{0}$
on $S^{1}$ has non-zero mean and satisfies the condition
\[
(\mu-\partial_{x}^{2})u_{0}\ge0\;(\mathrm{or}\text{\;\ensuremath{\le0)}}.
\]
Then the Cauchy problem \eqref{eq:muCH IVP} has a unique solution
$u\in\mathcal{C}^{\omega}(\mathbb{R}_{t}\times S_{x}^{1})$.

We have the following estimate of the radius of analyticity. Let $u_{0}\in A(r_{0})$.
Fix $\sigma_{0}<(\log r_{0})/(2\pi)$ and set 
\begin{align*}
K & =(20\pi d_{1}+8+4\pi^{2}c_{1})\left[1+\max\left\{ \|u(t)\|_{2};\,t\in[-T,T]\right\} \right],\\
\sigma(t) & =\sigma_{0}-\frac{\sqrt{6}\pi\gamma}{K}\|u_{0}\|_{(\sigma_{0},2)}(e^{K|t|/2}-1).
\end{align*}
Then, for any fixed $T>0$, we have $u(\cdot,t)\in A(e^{2\pi\sigma(t)})$
for $t\in[-T,T]$.
\end{thm}

There is an analogue about the $\mu$DP equation which reads as follows:
\begin{thm}
\label{thm:muDP main}Assume that a real-analytic function $u_{0}$
on $S^{1}$ has non-zero mean and satisfies the condition
\[
(\mu-\partial_{x}^{2})u_{0}\ge0\;(\mathrm{or}\text{\;\ensuremath{\le0)}}.
\]
Then the Cauchy problem \eqref{eq:muDP IVP} has a unique solution
$u\in\mathcal{C}^{\omega}(\mathbb{R}_{t}\times S_{x}^{1})$.

We have the following estimate of the radius of analyticity. Let $u_{0}\in A(r_{0})$.
Fix $\sigma_{0}<(\log r_{0})/(2\pi)$ and set 
\begin{align*}
K' & =(20\pi d_{1}+12)\left[1+\max\left\{ \|u(t)\|_{2};\,t\in[-T,T]\right\} \right],\\
\sigma'(t) & =\sigma_{0}-\frac{2\sqrt{2}\pi\gamma}{\sqrt{3}K'}\|u_{0}\|_{(\sigma_{0},2)}(e^{K'|t|/2}-1).
\end{align*}
Then, for any fixed $T>0$, we have $u(\cdot,t)\in A(e^{2\pi\sigma'(t)})$
for $t\in[-T,T]$.
\end{thm}

The proof of Theorem \ref{thm:muCH main} will be given in the later
subsections. First we will prove the analyticity in $x$ and establish
the lower bound of $\sigma(t)$. Next we will establish the analyticity
in $(t,x)$. The proof of Theorem \ref{thm:muDP main} is essentially
contained in that of Theorem \ref{thm:muCH main}. We can employ Remark
\ref{rem:DP KatoMasuda} instead of Proposition \ref{prop:Liapnov}.

\subsection{Regularity theorem by Kato-Masuda}

In \cite{Kato-Masuda}, the authors used their theory of Liapnov families
to prove a regularity result about the KdV and other equations. Later,
it was applied to a generalized Camassa-Holm equation in \cite{BHPglobal}.
Here we recall the abstract theorem in \cite{Kato-Masuda} in a weaker,
more concrete form. It is good enough for our purpose. The inequality
\eqref{eq:liapnov F} is a special case of that in \cite{Kato-Masuda}.
Notice that $v$ and $u(t,\cdot)$ are elements of $\mathcal{O}$
in \eqref{eq:liapnov F} and \eqref{eq:KMCP}. In applications, making
a suitable choice of the subset $\mathcal{O}$ of $Z$ is essential.
\begin{thm}
\label{thm:KatoMasuda weaker}Let $X$ and $Z$ be Banach spaces.
Assume that $Z$ is a dense subspace of $X$. Let $\mathcal{O}$ be
an open subset of $Z$ and $F$ be a continuous mapping from $Z$
to $X$. Let $\left\{ \Phi_{s};-\infty<s<\infty\right\} $ be a family
of real-valued functions on $Z$ satisfying

(a) The Fr\'echet partial derivative of $\Phi_{s}(v)$ in $v\in Z$
exists not only in $\mathcal{L}(Z;\mathbb{R})$ but also in $\mathcal{L}(X;\mathbb{R})$.
It is denoted by $D\Phi_{s}(v)$. This statement makes sense because
$\mathcal{L}(X;\mathbb{R})\subset\mathcal{L}(Z;\mathbb{R})$ by the
canonical identification. {[}(a) follows from (b) below.{]}

(b) The Fr\'echet derivative of $\Phi_{s}(v)$ in $(s,v)$ exists
not only in $\mathcal{L}(\mathbb{R}\times Z;\mathbb{R})$ but also
in $\mathcal{L}(\mathbb{R}\times X;\mathbb{R})$ and is continuous
from $\mathbb{R}\times Z$ to $\mathcal{L}(\mathbb{R}\times X;\mathbb{R})$.
This statement makes sense because $\mathcal{L}(\mathbb{R}\times X;\mathbb{R})\subset\mathcal{L}(\mathbb{R}\times Z;\mathbb{R})$
by the canonical identification.

(c) There exist positive constants $K$ and $L$ such that 
\begin{align}
\left|\left\langle F(v),D\Phi_{s}(v)\right\rangle \right| & \le K\Phi_{s}(v)+L\Phi_{s}(v)^{1/2}\partial_{s}\Phi_{s}(v)\label{eq:liapnov F}
\end{align}
holds for any $v\in\mathcal{O}.$ Here $\langle\cdot,\cdot\rangle$
(no subscript) is the pairing of $X$ and $\mathcal{L}(X;\mathbb{R})$.

Let $u\in\mathcal{C}([0,T];\,\mathcal{O})\cap\mathcal{C}^{1}([0,T];\,X)$
be the solution to the Cauchy problem
\begin{equation}
\begin{cases}
\dfrac{du}{dt}=F(u),\\
u(0,x)=u_{0}(x).
\end{cases}\label{eq:KMCP}
\end{equation}
Moreover, for a fixed constant $s_{0}$, set 
\begin{align*}
r(t) & =\Phi_{s_{0}}(u_{0})e^{Kt},\\
s(t) & =s_{0}-\int_{0}^{t}Lr(\tau)^{1/2}\,d\tau=s_{0}-\frac{2L\Phi_{s_{0}}(u_{0})^{1/2}}{K}(e^{Kt/2}-1),
\end{align*}
for $t\in[0,T]$. Then we have 
\[
\Phi_{s(t)}\left(u(t)\right)\le r(t),\;t\in[0,T].
\]
\end{thm}

Roughly speaking, this theorem means that the regularity of $u(t)$
for $t\in[0,T]$ follows from that of $u_{0}$. Later we will use
this when $X=H^{m+2},\,Z=H^{m+5}$ and $\Phi_{s}$ is related to some
variant of the Sobolev norms.

We can extend Theorem \ref{thm:KatoMasuda weaker} to $t\le0$.
\begin{cor}
\label{cor:KatoMasuda}Let $u\in\mathcal{C}([-T,T];\,\mathcal{O})\cap\mathcal{C}^{1}([-T,T];\,X)$
be the solution to the Cauchy problem \eqref{eq:KMCP}. Extend the
definition of $\rho$ and $\sigma$ sot that $r(t)=r(|t|),s(t)=s(|t|)$
by 
\begin{align*}
r(t) & =\Phi_{s_{0}}(u_{0})e^{K|t|},\\
s(t) & =s_{0}-\int_{0}^{|t|}L\rho(\tau)^{1/2}\,d\tau=s_{0}-\frac{2L\Phi_{s_{0}}(u_{0})^{1/2}}{K}(e^{K|t|/2}-1),
\end{align*}
 for $t\in[-T,T]$. Then we have 
\begin{equation}
\Phi_{s(t)}\left(u(t)\right)\le r(t),\;t\in[-T,T].\label{eq:KMconclusion}
\end{equation}
\end{cor}

\begin{proof}
For $t\le0$, set $t'=-t,\,\tilde{u}(t')=u(-t')=u(t)$. Then $\tilde{u}$
satisfies $d\tilde{u}/dt'=-F(\tilde{u}),\,\tilde{u}(0,x)=u_{0}(x)$.
Notice that $-F$ satisfies \eqref{eq:liapnov F}. By Theorem \ref{thm:KatoMasuda weaker},
we have $\Phi_{s(t')}(\tilde{u}(t'))\le r(t')$ for $t'\in[0,T]$.
It means $\Phi_{s(t)}(u(t))\le r(t)$ for $t\in[-T,0]$.
\end{proof}

\subsection{Pairing and estimates}

Recall the norm $\|\cdot\|_{(\sigma,s)}$ in Section $\ref{sec:Function-spaces}$.
When $s=2$, we have 
\[
\|v\|_{(\sigma,2)}^{2}=\sum_{j=0}^{\infty}\frac{1}{j!^{2}}e^{4\pi\sigma j}\|v^{(j)}\|_{2}^{2}.
\]
It is approximated by the finite sum
\[
\|v\|_{(\sigma,2,m)}^{2}=\sum_{j=0}^{m}\frac{1}{j!^{2}}e^{4\pi\sigma j}\|v^{(j)}\|_{2}^{2}.
\]
Set $\Psi_{j}(v)=\frac{1}{2}\|v^{(j)}\|_{2}^{2}=\frac{1}{2}\|\Lambda^{2}v^{(j)}\|_{0}^{2}=\frac{1}{2}\int_{S^{1}}\left(\Lambda^{2}v^{(j)}\right)^{2}\,dx$
for $j=0,1,\dots,m$. We introduce
\[
\Phi_{\sigma,m}(v)=\frac{1}{2}\|v\|_{(\sigma,2,m)}^{2}=\sum_{j=0}^{m}\frac{1}{j!^{2}}e^{4\pi\sigma j}\frac{\|v^{(j)}\|_{2}^{2}}{2}=\sum_{j=0}^{m}\frac{1}{j!^{2}}e^{4\pi\sigma j}\Psi_{j}(v).
\]
Later $\left\{ \Phi_{\sigma,m}\right\} $ will play the role of $\left\{ \Phi_{s}\right\} $
in Corollary \ref{cor:KatoMasuda}. Assume $w\in H^{m+2},v\in H^{m+5}.$
Let $D\Psi_{j}$ be the Fr\'echet derivative of $\Psi_{j}$ and $\langle w,D\Psi_{j}(v)\rangle$
be the pairing of $w\in H^{m+2}$ and $D\Psi_{j}(v)\in\left(H^{m+2}\right)^{*}\subset\left(H^{m+5}\right)^{*}$.
Indeed, $D\Psi_{j}(v)$ is an element of $\left(H^{m+2}\right)^{*}$,
because 
\begin{align}
\langle w,D\Psi_{j}(v)\rangle & =D\Psi_{j}(v)w=\left.\frac{d}{d\tau}\Psi_{j}(v+\tau w)\right|_{\tau=0}\label{eq:Frechet}\\
 & =\int_{S^{1}}\Lambda^{2}v^{(j)}\Lambda^{2}w^{(j)}\,dx=\langle v^{(j)},w^{(j)}\rangle_{2},\nonumber 
\end{align}
where $\langle\cdot,\cdot\rangle_{2}$ is the $H^{2}$ inner product.
Recall
\[
F_{\mu}(u)=-uu_{x}-\partial_{x}A^{-1}\left[2\mu(u)u+\frac{1}{2}u_{x}^{2}\right].
\]
It is easy to see that $F_{\mu}$ is continuous from $H^{m+5}$ to
$H^{m+2}$.
\begin{prop}
\label{prop:Liapnov}We have
\begin{align}
 & \left|\left\langle F_{\mu}(v),D\Phi_{\sigma,m}(v)\right\rangle \right|\label{eq:liapnov F_mu}\\
 & \le(20\pi d_{1}+8+4\pi^{2}c_{1})\|v\|_{2}\Phi_{\sigma,m}(v)+\sqrt{3}\pi\gamma\Phi_{\sigma,m}(v)^{1/2}\partial_{\sigma}\Phi_{\sigma,m}(v)\nonumber 
\end{align}
for $v\in H^{m+5}$, where $D\Phi_{\sigma,m}$ is the Fr\'echet derivative
of $\Phi_{\sigma,m}$.
\end{prop}

\begin{proof}
We have

\begin{align}
\langle F_{\mu}(v),D\Phi_{\sigma,m}(v)\rangle & =\sum_{j=0}^{m}\frac{1}{j!^{2}}e^{4\pi\sigma j}\langle F_{\mu}(v),D\Psi_{j}(v)\rangle.\label{eq:FvDPhi}
\end{align}
By \eqref{eq:Frechet}, we have
\begin{align}
 & \langle F_{\mu}(v),D\Psi_{j}(v)\rangle=\langle v^{(j)},\partial_{x}^{j}F_{\mu}(v)\rangle_{2}\label{eq:three terms}\\
 & =-\langle v^{(j)},\partial_{x}^{j}(vv_{x})\rangle_{2}-2\langle v^{(j)},\partial_{x}^{j+1}A^{-1}\left[\mu(v)v\right]\rangle_{2}-\frac{1}{2}\langle v^{(j)},\partial_{x}^{j+1}A^{-1}(v_{x}^{2})\rangle_{2}.\nonumber 
\end{align}

By using \eqref{eq:FvDPhi}, \eqref{eq:three terms} and the estimates
\eqref{eq:estimates 1}, \eqref{eq:estimates 2} and \eqref{eq:estimates 3}
below, we obtain \eqref{eq:liapnov F_mu}. In the following subsections,
we will calculate the three terms in \eqref{eq:three terms} separately.
\end{proof}
\begin{rem}
\label{rem:DP KatoMasuda} The $\mu$DP equation can be studied by
using the following estimate. Set 
\[
F_{\mu}^{\mathrm{DP}}(u)=-uu_{x}-\partial_{x}A^{-1}[3\mu(u)u].
\]
Then we have 
\[
|\langle F_{\mu}^{\mathrm{DP}}(v),D\Phi_{\sigma,m}(v)\rangle|\le(20\pi d_{1}+12)\|v\|_{2}\Phi_{\sigma,m}(v)+\frac{2\pi\gamma}{\sqrt{3}}\Phi_{\sigma,m}(v)^{1/2}\partial_{\sigma}\Phi_{\sigma,m}(v).
\]
\end{rem}

\subsubsection{Estimate of $\langle v^{(j)},\partial_{x}^{j}(vv_{x})\rangle_{2}$}

We have 
\begin{align}
 & \langle v^{(j)},\partial_{x}^{j}(vv_{x})\rangle_{2}=P_{j}+Q_{j},\label{eq:P_j and Q_j}\\
 & P_{j}=\langle v^{(j)},vv^{(j+1)}\rangle_{2},\label{eq:P_j}\\
 & Q_{j}=\sum_{\ell=1}^{j}\binom{j}{\ell}\langle v^{(j)},v^{(\ell)}v^{(j-\ell+1)}\rangle_{2}\,(j\ge1),\;Q_{0}=0.\label{eq:Q_j}
\end{align}
Because of \eqref{eq:H2norm}, we have 
\begin{equation}
P_{j}=I_{0}+\frac{1}{2\pi^{2}}I_{1}+\frac{1}{(2\pi)^{4}}I_{2},\label{eq:P_j-1}
\end{equation}
where $I_{i}=\langle v^{(j+i)},\partial_{x}^{i}(vv^{(j+1)})\rangle_{0}$.
Set $w=v^{(j)}$ for brevity. Then $I_{i}=\langle w^{(i)},\partial_{x}^{i}(vw_{x})\rangle_{0}$.
Integrating $\partial_{x}(vw^{2}/2)=vww_{x}+v_{x}w^{2}/2$, we obtain
\[
I_{0}=\langle w,vw_{x}\rangle_{0}=-\frac{1}{2}\int_{S^{1}}v_{x}w^{2}\,dx.
\]
Therefore the Schwarz inequality and Proposition \ref{prop:multiplication}
(ii) imply
\begin{align}
\left|I_{0}\right| & \le\frac{1}{2}\|v_{x}\|_{0}\|w^{2}\|_{0}\le\pi d_{1}\|v\|_{1}\|w\|_{0}\|w\|_{1}\le\pi d_{1}\|v\|_{2}\|w\|_{1}^{2}.\label{eq:I0}
\end{align}
Next, we have $I_{1}=\int_{S^{1}}w_{x}\partial_{x}(vw_{x})\,dx=-\int_{S^{1}}vw_{x}w_{xx}\,dx$
by integration by parts. Integrating $\partial_{x}(vw_{x}^{2}/2)=vw_{x}w_{xx}+v_{x}w_{x}^{2}/2$,
we obtain
\[
I_{1}=-\frac{1}{2}\int_{S^{1}}v_{x}w_{x}^{2}\,dx.
\]
Therefore we get, by the Schwarz inequality, Proposition \ref{prop:multiplication}
(ii) and Proposition \ref{prop:partial_x}, 
\begin{align}
\left|I_{1}\right| & \le\frac{1}{2}\|v_{x}\|_{0}\|w_{x}^{2}\|_{0}\le\pi d_{1}\|v\|_{1}\|w_{x}\|_{0}\|w_{x}\|_{1}\le4\pi^{3}d_{1}\|v\|_{2}\|w\|_{2}^{2}.\label{eq:I1}
\end{align}
The estimate of $I_{2}$ is a bit complicated. We divide it into a
sum of three terms as in 
\begin{align*}
I_{2} & =\langle w_{xx},\partial_{x}^{2}(vw_{x})\rangle_{0}=I_{21}+I_{22}+I_{23},
\end{align*}
where
\[
I_{21}=\langle w_{xx},v_{xx}w_{x}\rangle_{0},\,I_{22}=2\langle w_{xx},v_{x}w_{xx}\rangle_{0},\,I_{23}=\langle w_{xx},vw_{xxx}\rangle_{0}.
\]
We have 
\begin{align}
|I_{21}| & \le\|w_{xx}\|_{0}\|v_{xx}w_{x}\|_{0}\le d_{1}\|w_{xx}\|_{0}\|v_{xx}\|_{0}\|w_{x}\|_{1}\label{eq:I21}\\
 & \le32\pi^{5}d_{1}\|v\|_{2}\|w\|_{2}^{2},\nonumber \\
|I_{22}| & \le2\|w_{xx}\|_{0}\|v_{x}w_{xx}\|_{0}\le2d_{1}\|w_{xx}\|_{0}\|v_{x}\|_{1}\|w_{xx}\|_{0}\label{eq:I22}\\
 & \le64\pi^{5}d_{1}\|v\|_{2}\|w\|_{2}^{2}.\nonumber 
\end{align}
Integrating $\partial_{x}(vw_{xx}^{2}/2)=vw_{xx}w_{xxx}+v_{x}w_{xx}^{2}/2$,
we obtain
\begin{align*}
I_{23} & =\int_{S^{1}}vw_{xx}w_{xxx}\,dx=-\frac{1}{2}\int_{S^{1}}v_{x}w_{xx}^{2}\,dx=-\frac{1}{4}I_{22}.
\end{align*}
Hence 
\begin{equation}
|I_{23}|\le16\pi^{5}d_{1}\|v\|_{2}\|w\|_{2}^{2}.\label{eq:I23}
\end{equation}
The three inequalities \eqref{eq:I21}, \eqref{eq:I22} and \eqref{eq:I23}
give 
\begin{equation}
|I_{2}|\le112\pi^{5}d_{1}\|v\|_{2}\|w\|_{2}^{2}.\label{eq:I2}
\end{equation}
Combining \eqref{eq:P_j-1}, \eqref{eq:I0}, \eqref{eq:I1} and \eqref{eq:I2},
we obtain
\begin{equation}
|P_{j}|\le10\pi d_{1}\|v\|_{2}\|w\|_{2}^{2}=10\pi d_{1}\|v\|_{2}\|v^{(j)}\|_{2}^{2}\label{eq:I}
\end{equation}
and 
\begin{equation}
\left|\sum_{j=0}^{m}\frac{1}{j!^{2}}e^{4\pi\sigma j}P_{j}\right|\le10\pi d_{1}\|v\|_{2}\sum_{j=0}^{m}\frac{1}{j!^{2}}e^{4\pi\sigma j}\|v^{(j)}\|_{2}^{2}=20\pi d_{1}\|v\|_{2}\Phi_{\sigma,m}(v).\label{eq:sum P_j}
\end{equation}

Next we calculate $\sum_{j=1}^{m}j!^{-2}e^{4\pi\sigma j}Q_{j}$. Recall
\[
Q_{j}=\sum_{\ell=1}^{j}\binom{j}{\ell}\langle v^{(j)},v^{(\ell)}v^{(j-\ell+1)}\rangle_{2}\,(j\ge1),\;Q_{0}=0.
\]
By Proposition \ref{prop:multiplication} (iii), we have 
\begin{align*}
\|v^{(\ell)}v^{(j-\ell+1)}\|_{2} & \le\gamma\left(\|v^{(\ell)}\|_{2}\|v^{(j-\ell+1)}\|_{1}+\|v^{(\ell)}\|_{1}\|v^{(j-\ell+1)}\|_{2}\right)\\
 & \le2\pi\gamma\left(\|v^{(\ell)}\|_{2}\|v^{(j-\ell)}\|_{2}+\|v^{(\ell-1)}\|_{2}\|v^{(j-\ell+1)}\|_{2}\right).
\end{align*}
Combining this estimate with the Schwarz inequality, we get 
\begin{align*}
|Q_{j}| & \le2\pi\gamma(Q_{j,1}+Q_{j,2}),\\
Q_{j,1}= & \|v^{(j)}\|_{2}\sum_{\ell=1}^{j}\binom{j}{\ell}\|v^{(\ell)}\|_{2}\|v^{(j-\ell)}\|_{2},\\
Q_{j,2}= & \|v^{(j)}\|_{2}\sum_{\ell=1}^{j}\binom{j}{\ell}\|v^{(\ell-1)}\|_{2}\|v^{(j-\ell+1)}\|_{2}.
\end{align*}
Now we set $b_{k}=k!^{-1}e^{2\pi\sigma k}\|v^{(k)}\|_{2}\,(k=0,1,\dots,j)$.
Then we have
\begin{align*}
\frac{1}{j!^{2}}e^{4\pi\sigma j}Q_{j,1} & \le\sum_{\ell=1}^{j}b_{j}b_{\ell}b_{j-\ell},\\
\frac{1}{j!^{2}}e^{4\pi\sigma j}Q_{j,2} & \le\sum_{\ell=1}^{j}b_{j}\frac{b_{\ell-1}}{\ell}(j-\ell+1)b_{j-\ell+1}.
\end{align*}
Set $B=\left(\sum_{j=0}^{m}b_{j}^{2}\right)^{1/2},\widetilde{B}=\left(\sum_{j=1}^{m}jb_{j}^{2}\right)^{1/2}$.
Notice that $B^{2}=\|v\|_{(\sigma,2,m)}^{2}=2\Phi_{\sigma,m}(v),\,\widetilde{B}^{2}=(2\pi)^{-1}\partial_{\sigma}\Phi_{\sigma,m}(v)$.
The repeated use of the Schwarz inequality gives (\cite[Lemma 3.1]{Kato-Masuda})
\begin{align*}
 & \sum_{j=1}^{m}\frac{1}{j!^{2}}e^{4\pi\sigma j}Q_{j,1}\le\sum_{j=1}^{m}\sum_{\ell=1}^{j}b_{j}b_{\ell}b_{j-\ell}\le\sum_{\ell=1}^{m}\frac{b_{\ell}}{\sqrt{\ell}}\sum_{j=\ell}^{m}\sqrt{j}b_{j}b_{j-\ell}\\
 & \le B\widetilde{B}\sum_{\ell=1}^{m}\frac{\sqrt{\ell}b_{\ell}}{\ell}\le B\widetilde{B}^{2}\left(\sum_{\ell=1}^{m}\frac{1}{\ell^{2}}\right)^{1/2}\le\frac{\pi}{\sqrt{6}}B\widetilde{B}^{2},\\
 & \sum_{j=1}^{m}\frac{1}{j!^{2}}e^{4\pi\sigma j}Q_{j,2}\le\sum_{j=1}^{m}\sum_{\ell=1}^{j}b_{j}\frac{b_{\ell-1}}{\ell}(j-\ell+1)b_{j-\ell+1}\\
 & \le\sum_{\ell=1}^{m}\frac{b_{\ell-1}}{\ell}\sum_{j=\ell}^{m}\sqrt{j}b_{j}\sqrt{j-\ell+1}b_{j-\ell+1}\\
 & \le\widetilde{B}^{2}\sum_{\ell=1}^{m}\frac{b_{\ell-1}}{\ell}\le\widetilde{B}^{2}B\left(\sum_{\ell=1}^{m}\frac{1}{\ell^{2}}\right)^{1/2}\le\frac{\pi}{\sqrt{6}}B\widetilde{B}^{2}.
\end{align*}
Recalling $Q_{0}=0$, we obtain
\begin{equation}
\left|\sum_{j=0}^{m}\frac{1}{j!^{2}}e^{4\pi\sigma j}Q_{j}\right|\le\frac{(2\pi)^{2}\gamma}{\sqrt{6}}B\widetilde{B}^{2}=\frac{2\pi\gamma}{\sqrt{3}}\Phi_{\sigma,m}(v)^{1/2}\partial_{\sigma}\Phi_{\sigma,m}(v).\label{eq:sum Q_j}
\end{equation}
The inequalities \eqref{eq:sum P_j} and \eqref{eq:sum Q_j} give,
together with \eqref{eq:P_j and Q_j}, 
\begin{align}
 & \left|\sum_{j=0}^{m}\frac{1}{j!^{2}}e^{4\pi\sigma j}\langle v^{(j)},\partial_{x}^{j}(vv_{x})\rangle_{2}\right|\label{eq:estimates 1}\\
 & \le20\pi d_{1}\|v\|_{2}\Phi_{\sigma,m}(v)+\frac{2\pi\gamma}{\sqrt{3}}\Phi_{\sigma,m}(v)^{1/2}\partial_{\sigma}\Phi_{\sigma,m}(v).\nonumber 
\end{align}

\subsubsection{Estimate of $\langle v^{(j)},\partial_{x}^{j+1}A^{-1}\left[\mu(v)v\right]\rangle_{2}$}

We consider the second term of the right-hand side of \eqref{eq:three terms}.
By \eqref{eq:H2norm}, we have
\begin{align*}
\langle v^{(j)},\partial_{x}^{j+1}A^{-1}\left[\mu(v)v\right]\rangle_{2} & =\mu(v)\langle v^{(j)},A^{-1}v^{(j+1)}\rangle_{2}.\\
 & =\mu(v)\left(J_{0}+\frac{1}{2\pi^{2}}J_{1}+\frac{1}{(2\pi)^{4}}J_{2}\right),
\end{align*}
where $J_{i}=\langle v^{(j+i)},A^{-1}v^{(j+i+1)}\rangle_{0}\,(i=0,1,2)$.
Set $w=v^{(j)}$ for brevity. We have 
\[
J_{0}=\langle w,A^{-1}\partial_{x}w\rangle_{0},\,J_{1}=\langle w_{x},A^{-1}\partial_{x}^{2}w\rangle_{0},\,J_{2}=\langle w_{xx},A^{-1}\partial_{x}^{2}w_{x}\rangle_{0}.
\]
 Since $A^{-1}\partial_{x}^{i}:H^{0}\to H^{0}\,(i=0,1,2)$ are bounded
operators whose norms are equal to $1$, we have 
\begin{align*}
|\mu(v)J_{0}| & \le\|v\|_{0}\|w\|_{0}\|A^{-1}\partial_{x}w\|_{0}\le\|v\|_{2}\|w\|_{2}^{2},\\
|\mu(v)J_{1}| & \le\|v\|_{0}\|w_{x}\|_{0}\|A^{-1}\partial_{x}^{2}w\|_{0}\le2\pi\|v\|_{2}\|w\|_{2}^{2},\\
|\mu(v)J_{2}| & \le\|v\|_{0}\|w_{xx}\|_{0}\|A^{-1}\partial_{x}^{2}w_{x}\|_{0}\le(2\pi)^{3}\|v\|_{2}\|w\|_{2}^{2}.
\end{align*}
Therefore we get 
\begin{align*}
\left|\langle v^{(j)},\partial_{x}^{j+1}A^{-1}\left[\mu(v)v\right]\rangle_{2}\right| & \le[1+1/\pi+1/(2\pi)]\|v\|_{2}\|v^{(j)}\|_{2}^{2}\\
 & \le2\|v\|_{2}\|v^{(j)}\|_{2}^{2}
\end{align*}
 and 
\begin{equation}
\left|\sum_{j=0}^{m}\frac{1}{j!^{2}}e^{4\pi\sigma j}\langle v^{(j)},\partial_{x}^{j+1}A^{-1}\left[\mu(v)v\right]\rangle_{2}\right|\le4\|v\|_{2}\Phi_{\sigma,m}(v).\label{eq:estimates 2}
\end{equation}

\subsubsection{Estimate of $\langle v^{(j)},\partial_{x}^{j+1}A^{-1}(v_{x}^{2})\rangle_{2}$}

We consider the third term  of the right-hand side of \eqref{eq:three terms}.

First, assume $j=0$. Since the norm of $\partial_{x}A^{-1}\colon H^{1}\to H^{2}$
does not exceed 1, we have 
\begin{align}
 & \left|\langle v,\partial_{x}A^{-1}(v_{x}^{2})\rangle_{2}\right|\le\|v\|_{2}\|\partial_{x}A^{-1}(v_{x}^{2})\|_{2}\le\|v\|_{2}\|v_{x}^{2}\|_{1}\label{eq:estimate3-1}\\
 & \le c_{1}\|v\|_{2}\|v_{x}\|_{1}^{2}\le(2\pi)^{2}c_{1}\|v\|_{2}^{3}\le8\pi^{2}c_{1}\|v\|_{2}\Phi_{\sigma,m}(v).\nonumber 
\end{align}

Next we assume $j\ge1$. We have
\begin{align}
\langle v^{(j)},\partial_{x}^{j+1}A^{-1}(v_{x}^{2})\rangle_{2} & =\langle v^{(j)},(\partial_{x}^{2}A^{-1})\partial_{x}^{j-1}(v_{x}^{2})\rangle_{2}\label{eq:v^(j) pairing v_x^2}\\
 & =\sum_{\ell=0}^{j-1}\binom{j-1}{\ell}\langle v^{(j)},(\partial_{x}^{2}A^{-1})(v^{(\ell+1)}v^{(j-\ell)})\rangle_{2}\nonumber \\
 & =\sum_{\ell=1}^{j}\binom{j-1}{\ell-1}\langle v^{(j)},(\partial_{x}^{2}A^{-1})(v^{(\ell)}v^{(j-\ell+1)})\rangle_{2}.\nonumber 
\end{align}
This is similar to $Q_{j}$ in \eqref{eq:Q_j}. Since the norm of
$\partial_{x}^{2}A^{-1}\colon H^{2}\to H^{2}$ is 1, this operator
can be neglected in estimating $\langle v^{(j)},(\partial_{x}^{2}A^{-1})(v^{(\ell)}v^{(j-\ell+1)})\rangle_{2}$.
Moreover we have
\[
\binom{j-1}{\ell-1}\le\binom{j}{\ell}\,(j\ge1).
\]
We can follow \eqref{eq:sum Q_j} for $j\ge1$ and employ \eqref{eq:estimate3-1}
for $j=0$ (Recall $Q_{0}=0$). Finally we obtain 
\begin{align}
 & \left|\sum_{j=0}^{m}\frac{1}{j!^{2}}e^{4\pi\sigma j}\langle v^{(j)},\partial_{x}^{j+1}A^{-1}(v_{x}^{2})\rangle_{2}\right|\label{eq:estimates 3}\\
 & \le8\pi^{2}c_{1}\|v\|_{2}\Phi_{\sigma,m}(v)+\frac{2\pi\gamma}{\sqrt{3}}\Phi_{\sigma,m}(v)^{1/2}\partial_{\sigma}\Phi_{\sigma,m}(v).\nonumber 
\end{align}
It completes the proof of Proposition \ref{prop:Liapnov}.

\subsection{Analyticity in the space variable\label{subsec:Analyticity x}}

In this subsection, we prove a part of Theorem \ref{thm:muCH main}.
We assume that $u_{0}\in A(r_{0})$ is as in Theorem \ref{thm:muCH main}.
Then we can apply Theorem \ref{thm:KLM} with arbitrarily large $s$.
We will prove the analyticity of $u(t)$ in $x$ for each fixed $t$.
\begin{prop}
\label{prop:u analytic in x}Let $T>0$ be given. We have $u(t)\in A(e^{2\pi\sigma(t)})\subset A(e^{2\pi\sigma(T)})$
for $t\in[-T,T]$, where the function $\sigma(t)$ is defined in Theorem
\ref{thm:muCH main}.
\end{prop}

\begin{proof}
Theorem \ref{thm:KLM} implies $u(t)\in H^{\infty}$. Set 
\begin{align*}
\mu_{0} & =1+\max\left\{ \|u(t)\|_{2};\,t\in[-T,T]\right\} ,\\
\mathcal{O} & =\left\{ v\in H^{m+5};\,\|v\|_{2}<\mu_{0}\right\} ,\\
K & =(20\pi d_{1}+8+4\pi^{2}c_{1})\mu_{0}.
\end{align*}
Then $u(t)\in\mathcal{O}$ for $t\in[-T,T]$. Proposition \ref{prop:Liapnov}
implies that 
\begin{align}
\left|\left\langle F_{\mu}(v),D\Phi_{\sigma,m}(v)\right\rangle \right| & \le K\Phi_{\sigma,m}(v)+\sqrt{3}\pi\gamma\Phi_{\sigma,m}(v)^{1/2}\partial_{\sigma}\Phi_{\sigma,m}(v)\label{eq:liapnov F_mu-1}
\end{align}
holds for any $v\in\mathcal{O}.$ Fix $\sigma_{0}$ with $e^{2\pi\sigma_{0}}<r_{0}$.
Then $u_{0}\in A(r_{0})$ implies $\|u_{0}\|_{(\sigma_{0},2)}<\infty$
by Proposition \ref{prop:KatoMasudaLem2.2}.

Set 
\begin{align*}
\rho_{m}(t) & =\frac{1}{2}\|u_{0}\|_{(\sigma_{0},2,m)}^{2}e^{K|t|}=\Phi_{\sigma_{0},m}(u_{0})e^{K|t|},\\
\sigma_{m}(t) & =\sigma_{0}-\int_{0}^{|t|}\sqrt{3}\pi\gamma\rho_{m}(\tau)^{1/2}\,d\tau\\
 & =\sigma_{0}-\frac{\sqrt{6}\pi\gamma}{K}\|u_{0}\|_{(\sigma_{0},2,m)}(e^{K|t|/2}-1).
\end{align*}
These functions correspond to $r(t)$ and $s(t)$ in Corollary \ref{cor:KatoMasuda}.
Similarly, set
\begin{align*}
\rho(t) & =\frac{1}{2}\|u_{0}\|_{(\sigma_{0},2)}^{2}e^{K|t|},\\
\sigma(t) & =\sigma_{0}-\int_{0}^{|t|}\sqrt{3}\pi\gamma\rho(\tau)^{1/2}\,d\tau\\
 & =\sigma_{0}-\frac{\sqrt{6}\pi\gamma}{K}\|u_{0}\|_{(\sigma_{0},2)}(e^{K|t|/2}-1).
\end{align*}
We have $\rho_{m}(t)\le\rho_{m+1}(t)\le\rho(t)$, $\rho_{m}(t)\to\rho(t)\,(\text{as }m\to\infty)$
and $\sigma_{m}(t)\ge\sigma_{m+1}(t)\ge\sigma(t)$, $\sigma_{m}(t)\to\sigma(t)\,(\text{as }m\to\infty)$.
We can apply Corollary \ref{cor:KatoMasuda} and obtain
\begin{equation}
\Phi_{\sigma_{m}(t),m}(u(t))\le\rho_{m}(t)\le\rho(t),\quad t\in[-T,T].\label{eq:Liaponov conclusion}
\end{equation}
By Fatou's Lemma, 
\begin{align*}
\|u(t)\|_{(\sigma(t),2)} & =\sum_{j=0}^{\infty}\frac{1}{j!^{2}}e^{4\pi j\sigma(t)}\|\partial_{x}^{j}u(t)\|_{2}^{2}\le\liminf_{m\to\infty}\sum_{j=0}^{m}\frac{1}{j!^{2}}e^{4\pi j\sigma_{m}(t)}\|\partial_{x}^{j}u(t)\|_{2}^{2}\\
 & =2\liminf_{m\to\infty}\Phi_{\sigma_{m}(t),m}\left(u(t)\right)\le2\rho(t)<\infty.
\end{align*}
Therefore $u(t)\in A(e^{2\pi\sigma(t)})\subset A(e^{2\pi\sigma(T)})$
for $t\in[-T,T]$ by Proposition \ref{prop:KatoMasudaLem2.2}.
\end{proof}
\begin{prop}
The mapping $[-T,T]\to A(e^{2\pi\sigma(T)}),\,t\mapsto u(t)$ is continuous.
\end{prop}

\begin{proof}
Let $\left\{ t_{n}\right\} \subset[-T,T]$ be a sequence converging
to $t_{\infty}\in[-T,T]$. We have $u(t_{n})\to u(t_{\infty})$ in
$H^{\infty}$. On the other hand, $\|u(t_{n})\|_{(\sigma(T),2)}\,(n\ge0)$
is bounded since
\[
\|u(t)\|_{(\sigma(T),2)}\le\|u(t)\|_{(\sigma(t),2)}\le\rho(t)\le\rho(T).
\]
Proposition \ref{prop:KatoMasudaLem2.4} implies that $u(t_{n})$
converges to $u(t_{\infty})$ with respect to $\|\cdot\|_{(\sigma',2)}\,(\sigma'<\sigma(T))$.
By Corollary \ref{cor:Frechet}, this means convergence in $A(e^{2\pi\sigma(T)})$.
\end{proof}

\subsection{Analyticity in the space and time variables\label{subsec:Analyticity t, x}}

We continue the proof of Theorem \ref{thm:muCH main}. In the previous
subsection, we have established the analyticity in the space variable.
Here, we will prove the analyticity in the space and time variables.
By convention, a real-analytic function on a closed interval is real-analytic
on some open neighborhood of the closed interval.
\begin{prop}
Under the situation of Theorem \ref{thm:muCH main}, for any $T>0$,
there exists $\delta_{T}>0$ such that we have $u\in\mathcal{C}^{\omega}([-T,T],A(\delta_{T}))$.
\end{prop}

\begin{proof}
We have $u_{0}\in G^{\Delta,s+1}$ for any $\Delta<r_{0}$. Let $s>1/2$.
By Theorem \ref{thm:muCH IVP2}, there exists $\widetilde{T}>0$ such
that the Cauchy problem \eqref{eq:muCH IVP} has a unique solution
$\tilde{u}\in\mathcal{C}^{\omega}(|t|\le\widetilde{T}(1-d),G^{\Delta d,s+1})$
for $0<d<1$. We have $\tilde{u}=u$ by the local uniqueness, where
$u$ is the solution in Theorem \ref{thm:KLM}. Set $d=1/2,\,\widehat{T}=\widetilde{T}/2$.
Then $u=\tilde{u}\in\mathcal{C}^{\omega}(|t|\le\widehat{T},G^{\Delta/2,s+1})$.
By Proposition \ref{prop:embed G A}, a convergent series in $G^{\Delta/2,s+1}$
is convergent in $A(\Delta/2)$. We have $u\in\mathcal{C}^{\omega}(|t|\le\widehat{T},A(\Delta/2))$.

We have shown that $u$ is analytic in $t$ at least locally. Our
next step is to show that $u$ is analytic in $t$ globally. Set 
\begin{align*}
 & S=\left\{ T>0;\,u\in\mathcal{C}^{\omega}([-T,T],A(\delta_{T}))\;\text{for}\;\exists\delta_{T}>0\right\} \ni\widehat{T},\\
 & T^{*}=\sup S\ge\widehat{T}.
\end{align*}
We prove $T^{*}=\infty$ by contradiction. Assume $T^{*}<\infty$.
By Proposition \ref{prop:u analytic in x}, $u(T^{*})$ is well-defined
and there exists $\delta^{*}>0$ such that 
\[
u(T^{*})\in A(\delta^{*})\subset G^{\delta',s+1}\,(0<\delta'<\delta^{*}).
\]
 By Theorem \ref{thm:muCH IVP2} (with $t$ replaced with $t-T^{*}$),
there exists $\varepsilon>0$ and $\hat{u}\in\mathcal{C}^{\omega}([T^{*}-\varepsilon,T^{*}+\varepsilon],G^{\delta'/2,s+1})$
such that 
\begin{align*}
 & \hat{u}_{t}+\frac{1}{2}(\hat{u}^{2})_{x}+\partial_{x}A^{-1}\left[2\mu(\hat{u})\hat{u}+\frac{1}{2}\hat{u}_{x}^{2}\right]=0,\\
 & \hat{u}|_{t=T^{*}}=u(T^{*}).
\end{align*}
By the local uniqueness, we have $\hat{u}=u$. Namely, $\hat{u}$
is an extension of $u$ up to $|t|\le T^{*}+\varepsilon$ (valued
in $G^{\delta'/2,s+1}\subset A(\delta'/2)$). Therefore $T^{*}+\varepsilon\in S$.
This is a contradiction.
\end{proof}
\begin{prop}
\label{prop:global analyticity}Under the situation of Theorem \ref{thm:muCH main},
the Cauchy problem \eqref{eq:muCH IVP} has a unique solution $u\in\mathcal{C}^{\omega}(\mathbb{R}_{t}\times S_{x}^{1})$.
\end{prop}

\begin{proof}
The uniqueness in $H^{\infty}$ implies the uniqueness in the real-analytic
category.

Let $T$ be fixed. For $r>0$ sufficiently small, we have 
\[
\partial_{t}^{j}u(t)=\frac{j!}{2\pi i}\int_{|\tau-t|=r}\frac{u(\tau)}{(\tau-t)^{j+1}}\,d\tau,\;t\in[-T,T].
\]
The integral is performed in $A(\delta_{T})$ and converges with respect
to $\|\cdot\|_{(\sigma,0)}\,(e^{2\pi\sigma}<\delta_{T})$. By Cauchy's
estimate, there exists $C_{0}>1/r$ such that 
\begin{align*}
\|\partial_{t}^{j}u(t)\|_{(\sigma,0)} & \le C_{0}j!r^{-j}<C_{0}^{j+1}j!.
\end{align*}
Therefore we have 
\begin{align*}
\|\partial_{x}^{k}\partial_{t}^{j}u(\cdot,t)\|_{0} & \le C_{0}^{j+1}e^{-2\pi k\sigma}j!k!
\end{align*}
and there exists $C>0$ such that 
\begin{equation}
\left\Vert \partial_{x}^{k}\partial_{t}^{j}u\right\Vert _{L^{2}(S^{1}\times[-T,T])}\le\sqrt{2T}C^{j+k+1}(j+k)!.\label{eq:Komatsu}
\end{equation}
 Set $\Delta=\partial^{2}/\partial x^{2}+\partial^{2}/\partial t^{2}$.
The binomial expansion of $\Delta^{\ell}\,(\ell=0,1,2,\dots)$ and
\eqref{eq:Komatsu} yield
\begin{align*}
\left\Vert \Delta^{\ell}u\right\Vert _{L^{2}(S^{1}\times[-T,T])} & \le\sqrt{2T}C^{2\ell+1}(2\ell)!\sum_{p=0}^{\ell}\binom{\ell}{p}\\
 & \le\sqrt{2T}C^{2\ell+1}(2\ell)!2^{\ell}\le\sqrt{2T}C(\sqrt{2}C\ell)^{2\ell}.
\end{align*}
This estimates implies the real-analyticity of $u$ due to Theorem
\ref{thm:Komatsu} by Komatsu given below.
\end{proof}
In the last step of the proof of Proposition \ref{prop:global analyticity},
we have used the following theorem.
\begin{thm}
\label{thm:Komatsu}(\cite{Komatsu}) Let $\Omega$ be a domain in
$\mathbb{R}^{n}$ and let $P=P(\partial/\partial x_{1},\dots,\partial/\partial x_{n})$
be an elliptic partial differential operator of order $m$ with constant
coefficients. Then, for a function $f\in L_{\text{loc}}^{2}(\Omega)$
to be analytic in $\Omega$, it is (necessary and) sufficient that
1) for every $\ell\in\mathbb{Z}_{+}$, $P^{\ell}f$ (in the sense
of distributions) belongs to $L_{\text{loc}}^{2}(\Omega)$, and that
2) for every compact subset $K\subset\Omega$, there exist positive
constants $M$ and $A$ such that 
\begin{equation}
\|P^{\ell}f\|_{L^{2}(K)}\le M(A\ell)^{m\ell}.\label{eq:KomatsuTh}
\end{equation}
\end{thm}

A better known result in this direction is \cite{Kotake}, in which
$P=P(x,\partial/\partial x_{1},\dots,\partial/\partial x_{n})$, $x=(x_{1,},\dots,x_{n})$,
is an elliptic operator of order $m$ with analytic coefficients and
the right-hand side of \eqref{eq:KomatsuTh} is replaced with $M^{\ell+1}(m\ell)!$.
We can employ the result of \cite{Kotake} instead of Theorem \ref{thm:Komatsu}.

\section{Global-in-time solutions: higher-order case\label{sec:Global-in-time-solutions: higher}}

In this section we consider \eqref{eq:higher muCH IVP}. Global-in-time
solutions in Sobolev spaces have been studied in \cite{WLQ}. Notice
that the non-zero mean and the no-change-of-sign conditions are not
imposed.

\subsection{Statement of the main results}
\begin{thm}
(\cite[Theorem 2.1, 3.5]{WLQ})\label{thm:WLQ Theorem 2.1, 3.5} Let
$u_{0}\in H^{s}(S^{1})$, $s>7/2$. Then the Cauchy problem \eqref{eq:higher muCH IVP}
has a unique global solution $u$ in the space $\mathcal{C}(\mathbb{R},H^{s})\cap\mathcal{C}^{1}(\mathbb{R},H^{s-1})$.
Moreover, local uniqueness holds. 
\end{thm}

\begin{rem}
In \cite{WLQ}, the authors solve the Cauchy problem for $t>0$ only.
Since the equation is invariant under $(t,u)\mapsto(-t,-u)$, the
result for $t<0$ follows immediately.
\end{rem}

There is an analogous result about \eqref{eq:the other higher CP}.
\begin{thm}
(\cite[Theorem 2.1, 3.2]{WLQ2})\label{thm:WLQ2 Theorem 2.1, 3.2}
Let $u_{0}\in H^{s}(S^{1})$, $s>7/2$. Then the Cauchy problem \eqref{eq:the other higher CP}
has a unique global solution $u$ in the space $\mathcal{C}(\mathbb{R},H^{s})\cap\mathcal{C}^{1}(\mathbb{R},H^{s-1})$.
Moreover, local uniqueness holds.
\end{thm}

Our main result about the higher order $\mu$CH is the following theorem.
\begin{thm}
\label{thm:higher muCH main}If $u_{0}$ is a real-analytic function
on $S^{1}$, then the Cauchy problem \eqref{eq:higher muCH IVP} has
a unique solution $u\in\mathcal{C}^{\omega}(\mathbb{R}_{t}\times S_{x}^{1})$.

We have the following estimate of the radius of analyticity. Assume
$u_{0}\in A(r_{0})$. Fix $\sigma_{0}<(\log r_{0})/(2\pi)$ and set
\begin{align*}
\tilde{\sigma}(t) & =\sigma_{0}-\frac{\sqrt{2}\gamma_{2}}{\widetilde{K}}\|u_{0}\|_{(\sigma_{0},2)}(e^{\widetilde{K}t/2}-1),\\
\widetilde{K} & =\gamma_{1}\left[1+\max\left\{ \|u(t)\|_{4};\,t\in[-T,T]\right\} \right],\\
\gamma_{1} & =8+18\pi^{2}c_{1}+(20\pi+24\pi^{3}+104\pi^{4})d_{1},\\
\gamma_{2} & =\frac{16\pi\gamma}{\sqrt{3}}
\end{align*}
Then, for any fixed $T>0$, we have $u(\cdot,t)\in A(\tilde{\sigma}(t))$
for $t\in[-T,T]$.
\end{thm}

\begin{proof}
Theorem \ref{thm:WLQ Theorem 2.1, 3.5} implies $u(t)\in H^{\infty}$
if $u_{0}\in H^{\infty}$. Set 
\begin{align*}
\tilde{\mu}_{0} & =1+\max\left\{ \|u(t)\|_{4};\,t\in[-T,T]\right\} \,(H^{4}\text{ norm, not }H^{2}),\\
\widetilde{\mathcal{O}} & =\left\{ v\in H^{m+5};\,\|v\|_{4}<\tilde{\mu}{}_{0}\right\} \,(H^{4}\text{ norm, not }H^{2}),\\
\widetilde{K} & =\gamma_{1}\tilde{\mu}{}_{0}.\\
\tilde{\rho}(t) & =\frac{1}{2}\|u_{0}\|_{(\sigma_{0},2)}^{2}e^{\widetilde{K}t},\\
\tilde{\sigma}(t) & =\sigma_{0}-\int_{0}^{|t|}\gamma_{2}\rho(\tau)^{1/2}\,d\tau.
\end{align*}
Then the proof is almost the same as that of Theorem \ref{thm:muCH main}
and follows from Proposition \ref{prop:Liapnov-1} below. It is an
analogue of Proposition \ref{prop:Liapnov}. Notice that $\|v\|_{2}$
in \eqref{eq:liapnov F_mu} has been replaced with $\|v\|_{4}$.
\end{proof}
There is an analogous result about \eqref{eq:the other higher CP}.
\begin{thm}
\label{thm:higher muCH main-1}If $u_{0}$ is a real-analytic function
on $S^{1}$, then the Cauchy problem \eqref{eq:the other higher CP}
has a unique solution $u\in\mathcal{C}^{\omega}(\mathbb{R}_{t}\times S_{x}^{1})$.
The estimate of the radius of analyticity is the same as in Theorem
\ref{thm:higher muCH main}.
\end{thm}

The proof of Theorem \ref{thm:higher muCH main-1} is almost the same
as that of Theorem \ref{thm:higher muCH main}. One has only to replace
$B^{-1}$ with $A^{-2}$. The rest of this section is devoted to the
proof of Theorem \ref{thm:higher muCH main}. It is enough to prove
Proposition \ref{prop:Liapnov-1} below. 
\begin{prop}
\label{prop:Liapnov-1}We have 
\begin{align}
 & \left|\left\langle G(v),D\Phi_{\sigma,m}(v)\right\rangle \right|\label{eq:liapnov F_mu-2}\\
 & \le\gamma_{1}\|v\|_{4}\Phi_{\sigma,m}(v)+\gamma_{2}\Phi_{\sigma,m}(v)^{1/2}\partial_{\sigma}\Phi_{\sigma,m}(v),\nonumber 
\end{align}
where $\gamma_{1}$ and $\gamma_{2}$ are given in Theorem \ref{thm:higher muCH main}.
\end{prop}

\begin{proof}
Recall \eqref{eq:G}, namely
\[
G(u)=-uu_{x}-\partial_{x}B^{-1}\left[2\mu(u)u+\frac{1}{2}u_{x}^{2}-3u_{x}u_{xxx}-\frac{7}{2}u_{xx}^{2}\right].
\]
We have

\begin{align}
\langle G(v),D\Phi_{\sigma,m}(v)\rangle & =\sum_{j=0}^{m}\frac{1}{j!^{2}}e^{4\pi\sigma j}\langle G(v),D\Psi_{j}(v)\rangle.\label{eq:FvDPhi-1}
\end{align}
By \eqref{eq:Frechet}, we have
\begin{align}
 & \langle G(v),D\Psi_{j}(v)\rangle=\langle v^{(j)},\partial_{x}^{j}G(v)\rangle_{2}\label{eq:five terms}\\
 & =-\langle v^{(j)},\partial_{x}^{j}(vv_{x})\rangle_{2}-2\langle v^{(j)},\partial_{x}^{j+1}B^{-1}\left[\mu(v)v\right]\rangle_{2}-\frac{1}{2}\langle v^{(j)},\partial_{x}^{j+1}B^{-1}(v_{x}^{2})\rangle_{2}\nonumber \\
 & \quad+3\langle v^{(j)},\partial_{x}^{j+1}B^{-1}\left[v_{x}v_{xxx}\right]\rangle_{2}+\frac{7}{2}\langle v^{(j)},\partial_{x}^{j+1}B^{-1}(v_{xx}^{2})\rangle_{2}.\nonumber 
\end{align}
We compare it with \eqref{eq:three terms}. We encountered $\langle v^{(j)},\partial_{x}^{j}(vv_{x})\rangle_{2}$
in \eqref{eq:three terms}. The following two terms $-2\langle v^{(j)},\partial_{x}^{j+1}B^{-1}\left[\mu(v)v\right]\rangle_{2}$
and $-\frac{1}{2}\langle v^{(j)},\partial_{x}^{j+1}B^{-1}(v_{x}^{2})\rangle_{2}$
have better estimates than $-2\langle v^{(j)},\partial_{x}^{j+1}A^{-1}\left[\mu(v)v\right]\rangle_{2}$
and $-\frac{1}{2}\langle v^{(j)},\partial_{x}^{j+1}A^{-1}(v_{x}^{2})\rangle_{2}$.
Therefore we can employ \eqref{eq:estimates 1} and analogues of \eqref{eq:estimates 2}
and \eqref{eq:estimates 3}. Here we replace $\|v\|_{2}$ with $(\|v\|_{2}\le)\|v\|_{4}$.
Our remaining task is to estimate $3\langle v^{(j)},\partial_{x}^{j+1}B^{-1}\left[v_{x}v_{xxx}\right]\rangle_{2}$
and $\frac{7}{2}\langle v^{(j)},\partial_{x}^{j+1}B^{-1}(v_{xx}^{2})\rangle_{2}$.
The results will be given as \eqref{eq:pairing v_x v_xxx j=00003D0 to m}
and \eqref{eq:pair v_xx^2  0m} in the following subsection.
\end{proof}

\subsection{Estimates: higher-order case}

\subsubsection{Estimate of $\langle v^{(j)},\partial_{x}^{j+1}B^{-1}\left[v_{x}v_{xxx}\right]\rangle_{2}$}

First we assume $j\ge3$. Since 
\begin{align*}
\langle v^{(j)},\partial_{x}^{j+1}B^{-1}\left[v_{x}v_{xxx}\right]\rangle_{2} & =\langle v^{(j)},(\partial_{x}^{4}B^{-1})\partial_{x}^{j-3}\left[v_{x}v_{xxx}\right]\rangle_{2}\\
 & =\sum_{\ell=0}^{j-3}\binom{j-3}{\ell}\langle v^{(j)},(\partial_{x}^{4}B^{-1})v^{(\ell+1)}v^{(j-\ell)}\rangle_{2}\\
 & =\sum_{\ell=1}^{j-2}\binom{j-3}{\ell-1}\langle v^{(j)},(\partial_{x}^{4}B^{-1})v^{(\ell)}v^{(j-\ell+1)}\rangle_{2}.
\end{align*}
This is better than \eqref{eq:Q_j}. Indeed, $\partial_{x}^{4}B^{-1}\colon H^{2}\to H^{2}$
is as good as $\partial_{x}^{2}A^{-1}\colon H^{2}\to H^{2}$ and the
binomial coefficients have become smaller. We follow \eqref{eq:sum Q_j}
with $j\ge3$ instead of $j\ge0$ and get
\begin{equation}
\left|\sum_{j=3}^{m}\frac{1}{j!^{2}}e^{4\pi\sigma j}\langle v^{(j)},\partial_{x}^{j+1}B^{-1}(v_{x}v_{xxx})\rangle_{2}\right|\le\frac{2\pi\gamma}{\sqrt{3}}\Phi_{\sigma,m}(v)^{1/2}\partial_{\sigma}\Phi_{\sigma,m}(v).\label{eq:pairing v_x v_xxx j >=00003D3}
\end{equation}
Next we consider the case $j=0$. We have
\[
|\langle v,\partial_{x}B^{-1}[v_{x}v_{xxx}]\rangle|_{2}\le\|v\|_{2}\|\partial_{x}B^{-1}[v_{x}v_{xxx}]\|_{2}
\]
and, by Proposition \ref{prop:multiplication} (ii),
\begin{align*}
\|\partial_{x}B^{-1}[v_{x}v_{xxx}]\|_{2} & \le\|v_{x}v_{xxx}\|_{0}\le d_{1}\|v_{x}\|_{1}\|v_{xxx}\|_{0}\\
 & \le(2\pi)^{4}d_{1}\|v\|_{2}\|v\|_{3}.
\end{align*}
These two inequalities yield
\begin{equation}
|\langle v,\partial_{x}B^{-1}[v_{x}v_{xxx}]\rangle|_{2}\le(2\pi)^{4}d_{1}\|v\|_{3}\|v\|_{2}^{2}\le(2\pi)^{4}d_{1}\|v\|_{4}\|v\|_{2}^{2}.\label{eq:pairing v_x v_xxx j=00003D0}
\end{equation}
We have used $\|v\|_{3}\le\|v\|_{4}$, because $\|v\|_{4}$ will inevitably
appear later. We deal with it by modifying the definition of $\mathcal{O}$,
so that $\|v\|_{3}$ and $\|v\|_{4}$ are bounded there. Proposition
\ref{prop:Liapnov} should be modified accordingly.

Next assume $j=1$. We have
\[
|\langle v^{(1)},\partial_{x}^{2}B^{-1}[v_{x}v_{xxx}]\rangle|_{2}\le\|v^{(1)}\|_{2}\|\partial_{x}^{2}B^{-1}[v_{x}v_{xxx}]\|_{2}
\]
and, by Proposition \ref{prop:multiplication} (ii) again, 
\begin{align*}
\|\partial_{x}^{2}B^{-1}[v_{x}v_{xxx}]\|_{2} & \le\|v_{x}v_{xxx}\|_{0}\le d_{1}\|v_{x}\|_{1}\|v_{xxx}\|_{0}\\
 & \le(2\pi)^{3}d_{1}\|v\|_{2}\|v^{(1)}\|_{2}.
\end{align*}
Therefore 
\begin{align}
\left|e^{4\pi\sigma}\langle v^{(1)},\partial_{x}^{2}B^{-1}[v_{x}v_{xxx}]\rangle_{2}\right| & \le(2\pi)^{3}d_{1}\|v\|_{2}e^{4\pi\sigma}\|v^{(1)}\|_{2}^{2}\label{eq:pairing v_x v_xxx j=00003D1}\\
 & \le(2\pi)^{3}d_{1}\|v\|_{4}e^{4\pi\sigma}\|v^{(1)}\|_{2}^{2}.\nonumber 
\end{align}
Next we assume $j=2$. We have
\[
|\langle v^{(2)},\partial_{x}^{3}B^{-1}[v_{x}v_{xxx}]\rangle|_{2}\le\|v^{(2)}\|_{2}\|\partial_{x}^{3}B^{-1}[v_{x}v_{xxx}]\|_{2}\le(2\pi)^{2}\|v\|_{4}\|\partial_{x}^{3}B^{-1}[v_{x}v_{xxx}]\|_{2}
\]
and, by Proposition \ref{prop:multiplication} (ii), 
\begin{align*}
\|\partial_{x}^{3}B^{-1}[v_{x}v_{xxx}]\|_{2} & \le\|\partial_{x}(v_{x}v_{xxx})\|_{0}\le\|v^{(2)}v^{(3)}+v^{(1)}v^{(4)}\|_{0}\\
 & \le d_{1}(\|v^{(2)}\|_{1}\|v^{(3)}\|_{0}+\|v^{(4)}\|_{0}\|v^{(1)}\|_{1})\le(2\pi+4\pi^{2})d_{1}\|v^{(2)}\|_{2}^{2}.
\end{align*}
Therefore
\begin{equation}
\left|e^{8\pi\sigma}\langle v^{(2)},\partial_{x}^{3}B^{-1}[v_{x}v_{xxx}]\rangle_{2}\right|\le(8\pi^{3}+16\pi^{4})d_{1}\|v\|_{4}e^{8\pi\sigma}\|v^{(2)}\|_{2}^{2}.\label{eq:pairing v_x v_xxx j=00003D2}
\end{equation}

By \eqref{eq:pairing v_x v_xxx j >=00003D3}, \eqref{eq:pairing v_x v_xxx j=00003D0},
\eqref{eq:pairing v_x v_xxx j=00003D1} and \eqref{eq:pairing v_x v_xxx j=00003D2},
we obtain
\begin{align}
 & \left|\sum_{j=0}^{m}\frac{1}{j!^{2}}e^{4\pi\sigma j}\langle v^{(j)},\partial_{x}^{j+1}B^{-1}(v_{x}v_{xxx})\rangle_{2}\right|\label{eq:pairing v_x v_xxx j=00003D0 to m}\\
 & \le(8\pi^{3}+16\pi^{4})d_{1}\|v\|_{4}\Phi_{\sigma,m}(v)+\frac{2\pi\gamma}{\sqrt{3}}\Phi_{\sigma,m}(v)^{1/2}\partial_{\sigma}\Phi_{\sigma,m}(v).\nonumber 
\end{align}

\subsubsection{Estimate of $\langle v^{(j)},\partial_{x}^{j+1}B^{-1}\left[v_{xx}^{2}\right]\rangle_{2}$}

First we assume $j\ge3$. Since 
\begin{align*}
\langle v^{(j)},\partial_{x}^{j+1}B^{-1}\left[v_{xx}^{2}\right]\rangle_{2} & =\langle v^{(j)},(\partial_{x}^{4}B^{-1})\partial_{x}^{j-3}\left[v_{xx}^{2}\right]\rangle_{2}\\
 & =\sum_{\ell=0}^{j-3}\binom{j-3}{\ell}\langle v^{(j)},(\partial_{x}^{4}B^{-1})v^{(\ell+2)}v^{(j-\ell-1)}\rangle_{2}\\
 & =\sum_{\ell=2}^{j-1}\binom{j-3}{\ell-2}\langle v^{(j)},(\partial_{x}^{4}B^{-1})v^{(\ell)}v^{(j-\ell+1)}\rangle_{2}.
\end{align*}
This is better than \eqref{eq:Q_j}. We follow \eqref{eq:sum Q_j}
with $j\ge3$ instead of $j\ge0$ and get
\begin{equation}
\left|\sum_{j=3}^{m}\frac{1}{j!^{2}}e^{4\pi\sigma j}\langle v^{(j)},\partial_{x}^{j+1}B^{-1}(v_{xx}^{2})\rangle_{2}\right|\le\frac{2\pi\gamma}{\sqrt{3}}\Phi_{\sigma,m}(v)^{1/2}\partial_{\sigma}\Phi_{\sigma,m}(v).\label{eq:pairing v_xx^2  j >=00003D3}
\end{equation}
Next we consider the case $j=0$. We have
\[
|\langle v,\partial_{x}B^{-1}[v_{xx}^{2}]\rangle|_{2}\le\|v\|_{2}\|\partial_{x}B^{-1}[v_{xx}^{2}]\|_{2}
\]
and, by Proposition \ref{prop:multiplication} (ii),
\begin{align*}
\|\partial_{x}B^{-1}[v_{xx}^{2}]\|_{2} & \le\|v_{xx}^{2}\|_{0}\le d_{1}\|v_{xx}\|_{1}\|v_{xx}\|_{0}\\
 & \le(2\pi)^{4}d_{1}\|v\|_{3}\|v\|_{2}.
\end{align*}
These two inequalities yield
\begin{equation}
|\langle v,\partial_{x}B^{-1}[v_{xx}^{2}]\rangle|_{2}\le(2\pi)^{4}d_{1}\|v\|_{3}\|v\|_{2}^{2}\le(2\pi)^{4}d_{1}\|v\|_{4}\|v\|_{2}^{2}.\label{eq:pairing v_xx^2  j=00003D0}
\end{equation}

Next assume $j=1$. We have
\[
|\langle v^{(1)},\partial_{x}^{2}B^{-1}[v_{xx}^{2}]\rangle_{2}|\le\|v^{(1)}\|_{2}\|\partial_{x}^{2}B^{-1}[v_{xx}^{2}]\|_{2}
\]
and, by Proposition \ref{prop:multiplication} (ii), 
\begin{align*}
\|\partial_{x}^{2}B^{-1}[v_{xx}^{2}]\|_{2} & \le\|v_{xx}^{2}\|_{0}\le d_{1}\|v_{xx}\|_{1}\|v_{xx}\|_{0}\\
 & \le(2\pi)^{3}d_{1}\|v^{(1)}\|_{2}\|v\|_{2}.
\end{align*}
Therefore 
\begin{align}
\left|e^{4\pi\sigma}\langle v^{(1)},\partial_{x}^{2}B^{-1}[v_{xx}^{2}]\rangle_{2}\right| & \le(2\pi)^{3}d_{1}\|v\|_{2}e^{4\pi\sigma}\|v^{(1)}\|_{2}^{2}\label{eq:pairing v_xx^2  j=00003D1}\\
 & \le(2\pi)^{3}d_{1}\|v\|_{4}e^{4\pi\sigma}\|v^{(1)}\|_{2}^{2}.\nonumber 
\end{align}
Next assume $j=2$. We have
\begin{align*}
|\langle v^{(2)},\partial_{x}^{3}B^{-1}[v_{xx}^{2}]\rangle_{2}| & \le\|v^{(2)}\|_{2}\|(\partial_{x}^{3}B^{-1})[v_{xx}^{2}]\|_{2}\\
 & \le(2\pi)^{2}\|v\|_{4}\|(\partial_{x}^{3}B^{-1})[v_{xx}^{2}]\|_{2}
\end{align*}
and 
\begin{align*}
\|(\partial_{x}^{3}B^{-1})[v_{xx}^{2}]\|_{2} & \le\|v_{xx}^{2}\|_{1}\le c_{1}\|v_{xx}\|_{1}^{2}\le c_{1}\|v^{(2)}\|_{2}^{2}.
\end{align*}
Therefore
\begin{equation}
\left|e^{8\pi\sigma}\langle v^{(2)},\partial_{x}^{3}B^{-1}[v_{xx}^{2}]\rangle_{2}\right|\le(2\pi)^{2}c_{1}\|v\|_{4}e^{8\pi\sigma}\|v^{(2)}\|_{2}^{2}.\label{eq:pairing v_xx^2  j=00003D2}
\end{equation}
By using \eqref{eq:pairing v_xx^2  j >=00003D3}, \eqref{eq:pairing v_xx^2  j=00003D0},
\eqref{eq:pairing v_xx^2  j=00003D1} and \eqref{eq:pairing v_xx^2  j=00003D2},
we obtain
\begin{align}
 & \left|\sum_{j=0}^{m}\frac{1}{j!^{2}}e^{4\pi\sigma j}\langle v^{(j)},\partial_{x}^{j+1}B^{-1}(v_{xx}^{2})\rangle_{2}\right|\label{eq:pair v_xx^2  0m}\\
 & \le(16\pi^{4}d_{1}+4\pi^{2}c_{1})\|v\|_{4}\Phi_{\sigma,m}(v)+\frac{2\pi\gamma}{\sqrt{3}}\Phi_{\sigma,m}(v)^{1/2}\partial_{\sigma}\Phi_{\sigma,m}(v).\nonumber 
\end{align}

\section*{Appendix: Local study of the non-quasilinear modified $\mu$CH equation\label{sec:Appendix: local-study}}

The difficulty of  \eqref{eq:modifiedmuCH2} lies in the presence
of the non-quasilinear terms $-\frac{1}{3}u_{x}^{3}$ and $-\frac{1}{3}\mu(u_{x}^{3})$.
In \cite{BHPpower}, the authors employed the power series method
to deal with the non-quasilinear term $au^{k-2}u_{x}^{3}$ of the
$k$-$abc$-equation. In the present paper, we overcome the difficulty
of non-quasilinearity by a $\mu$-version of a classical trick used
in the proof of the Cauchy-Kowalevsky theorem (\cite{CourantHilbert})
following \cite{HM} and \cite{FuYing}. We set $v=u_{x}$ and differentiate
 \eqref{eq:modifiedmuCH2} in $x$. It can be proved that  \eqref{eq:modifiedmuCH2}
is equivalent to the following quasilinear modified $\mu$CH system:
\begin{equation}
\begin{cases}
u_{t}+2\mu(u)uv-\dfrac{v^{3}}{3}+\partial_{x}A^{-1}\left[2\mu^{2}(u)u+\mu(u)v^{2}+\gamma u\right]+\dfrac{\mu(v^{3})}{3}=0, & \hspace{-2em}\\
v_{t}+2\mu(u)(uv)_{x}-\dfrac{(v^{3})_{x}}{3}+\partial_{x}A^{-1}\left[2\mu^{2}(u)v+\mu(u)(v^{2})_{x}+\gamma v\right]=0. & \hspace{-2em}
\end{cases}\label{eq:modifiedmuCHsystem}
\end{equation}

Of course, this trick works for the the $k$-$abc$-equation as well.
We can prove unique solvability of the Cauchy problem for \eqref{eq:modifiedmuCHsystem}
and  \eqref{eq:modifiedmuCH2}.

The Cauchy problem \eqref{eq:modifiedmuCH2} for the non-quasilinear
modified $\mu$CH equation can be written in the following form: 
\begin{equation}
\begin{cases}
u_{t}+2\mu(u)uv-\frac{1}{3}v^{3}+\partial_{x}A^{-1}\left[2\mu^{2}(u)u+\mu(u)v^{2}+\gamma u\right]+\frac{1}{3}\mu(v^{3})=0,\\
v=u_{x},\\
u(0,x)=u_{0}(x).
\end{cases}\label{eq:modifiedmuCH-1}
\end{equation}

We introduce the system below.
\begin{equation}
\begin{cases}
u_{t}+2\mu(u)uv-\dfrac{1}{3}v^{3}+\partial_{x}A^{-1}\left[2\mu^{2}(u)u+\mu(u)v^{2}+\gamma u\right]+\dfrac{1}{3}\mu(v^{3})=0, & \hspace{-2em}\\
v_{t}+2\mu(u)(uv)_{x}-\dfrac{1}{3}(v^{3})_{x}+\partial_{x}A^{-1}\left[2\mu^{2}(u)v+\mu(u)(v^{2})_{x}+\gamma v\right]=0, & \hspace{-2em}\\
u(0,x)=u_{0}(x),v(0,x)=v_{0}(x). & \hspace{-2em}
\end{cases}\label{eq:modifiedmuCHsystem-1}
\end{equation}

\begin{thm}
\label{thm:equivalence}The Cauchy problems  \eqref{eq:modifiedmuCH-1}
and  \eqref{eq:modifiedmuCHsystem-1} are equivalent to each other
if $v_{0}(x)=\partial_{x}u_{0}(x)$. In particular, \eqref{eq:modifiedmuCHsystem-1}
implies $v=\partial_{x}u=u_{x}$ if $v_{0}(x)=\partial_{x}u_{0}(x)$.
\end{thm}

\begin{proof}
By differentiation with respect to $x$, we get the second equation
in  \eqref{eq:modifiedmuCHsystem-1} from \eqref{eq:modifiedmuCH-1}.

To show the converse, differentiate both sides of the first equation
of  \eqref{eq:modifiedmuCHsystem-1} in $x$. By comparing it with
the second equation, we get
\[
(v-u_{x})_{t}+\partial_{x}A^{-1}\left[2\mu^{2}(u)(v-u_{x})+\gamma(v-u_{x})\right]=0.
\]
 It is enough to prove that $w_{t}+\partial_{x}A^{-1}[a(t)w]=0$ and
$w(0,x)=0$ imply $w=0$, where $a(t)$ is a continuous function in
$t$. Set $w=\sum_{k\in\mathbb{Z}}w_{k}(t)e^{2k\pi ix}$. Then we
have 
\[
w_{0}'(t)=0,\;w_{k}'(t)-\frac{a(t)w_{k}(t)}{2k\pi i}=0\,(k\ne0).
\]
Since $w_{k}(0)=0$, we have $w_{k}(t)=0$ ($k\in\mathbb{Z}$) for
any $t$. It implies $w=0$ and $v=u_{x}.$
\end{proof}
\begin{rem}
The trick of setting $v=u_{x}$ has been used in \cite{HM} (non-$\mu$
equation) and \cite{FuYing} in a different function space rather
formally, i.e. without discussion corresponding to Theorem \ref{thm:equivalence}.
In \cite{HM}, this trick is applied to a quasilinear equation.
\end{rem}

\begin{thm}
\label{thm:system} If $u_{0}$ and $v_{0}$ are real-analytic functions
on $S^{1}$, then the Cauchy problem  \eqref{eq:modifiedmuCHsystem-1}
has a holomorphic solution near $t=0$. More precisely, we have the
following:

(i) There exists $\Delta>0$ such that $u_{0},v_{0}\in G^{\Delta,s+1}\subset A(\Delta)$
for any $s$.

(ii) If $s>1/2$, there exists a positive time $T_{\Delta}=T(u_{0},v_{0},s,\Delta)$
such that for every $d\in(0,1)$, the Cauchy problem \eqref{eq:modifiedmuCHsystem-1}
has a unique solution which is a holomorphic function valued in $\oplus^{2}G^{\Delta d,s+1}$
in the disk $D(0,T_{\Delta}(1-d))$. Furthermore, if $\gamma=0$,
the analytic lifespan $T_{\Delta}$ satisfies 
\[
T_{\Delta}=\frac{\mathrm{const.}}{\|(u_{0},u_{0}')\|_{\Delta,s+1}^{2}},
\]
when $\Delta$ is fixed. On the other hand, if $\gamma\ne0$, we have
the asymptotic behavior
\[
T_{\Delta}\approx\frac{\mathrm{const.}}{\|(u_{0},v_{0})\|_{\Delta,s+1}^{2}}\;(\text{large initial values}),\quad T_{\Delta}\approx\mathrm{const.}\;(\text{small initial values}).
\]
\end{thm}

\begin{proof}
We give a detailed proof assuming $\Delta=1$. The general case can
be proved in the same way as for Theorem \ref{thm:muCH IVP2}. The
norm on $\oplus^{2}G^{\delta,s+1}$ is defined by $\|(u,v)\|_{\delta,s+1}=\|u\|_{\delta,s+1}+\|v\|_{\delta,s+1}$.
We use the same notation $\|\cdot\|_{\delta,s+1}$ for $G^{\delta,s+1}$
and $\oplus^{2}G^{\delta,s+1}$. Assume $\|(u,v)-(u_{0},v_{0})\|_{\delta,s+1}<R,\|(u',v')-(u_{0},v_{0})\|_{\delta,s+1}<R$.
Set $R_{s+1}=R+\max\left(\|u_{0}\|_{1,s+1},\|v_{0}\|_{1,s+1}\right)$.
Assume $0<\delta'<\delta\le1$. Then by \eqref{eq:ball1}, we have
\begin{equation}
\|w\|_{\delta,s+1}<R_{s+1},\;|\mu(w)|<R_{s+1},\label{eq:R_=00007Bs+1=00007D}
\end{equation}
where $w$ is any of $u,v,u',v'$. Let us consider differences concerning
$\mu(u)uv$ and $\mu(u)(uv)_{x}$. Since $uv-u'v'=u(v-v')+(u-u')v'$,
we get 
\begin{align}
 & \|uv-u'v'\|_{\delta',s+1}\le\|uv-u'v'\|_{\delta,s+1}\label{eq:uv}\\
 & \le c_{s+1}\left(\|u\|_{\delta,s+1}\|v-v'\|_{\delta,s+1}+\|u-u'\|_{\delta,s+1}\|v'\|_{\delta,s+1}\right)\nonumber \\
 & \le c_{s+1}R_{s+1}\|(u,v)-(u',v')\|_{\delta,s+1}.\nonumber 
\end{align}
Combining \eqref{eq:uv} with Proposition \ref{prop:partial_x}, we
get 
\begin{align}
\|(uv)_{x}-(u'v')_{x}\|_{\delta',s+1} & \le\frac{e^{-1}}{\delta-\delta'}\|(uv)_{x}-(u'v')_{x}\|_{\delta,s+1}\label{eq:uv-1}\\
 & \le\frac{e^{-1}c_{s+1}R_{s+1}}{\delta-\delta'}\|(u,v)-(u',v')\|_{\delta,s+1}.\nonumber 
\end{align}

On the other hand, we have 
\begin{equation}
\|(u'v')_{x}\|_{\delta',s+1}\le\frac{e^{-1}}{\delta-\delta'}\|u'v'\|_{\delta,s+1}\le\frac{e^{-1}c_{s+1}R_{s+1}^{2}}{\delta-\delta'}.\label{eq:u'v'_x}
\end{equation}

Combining 
\begin{align*}
\mu(u)uv-\mu(u')u'v' & =\mu(u)(uv-u'v')+\left[\mu(u)-\mu(u')\right]u'v'
\end{align*}
with \eqref{eq:R_=00007Bs+1=00007D} and  \eqref{eq:uv}, we obtain
\begin{align}
 & \|\mu(u)uv-\mu(u')u'v'\|_{\delta',s+1}\label{eq:mu u uv}\\
 & \le|\mu(u)|\|uv-u'v'\|_{\delta',s+1}+|\mu(u)-\mu(u')|\|u'v'\|_{\delta',s+1}\nonumber \\
 & \le R_{s+1}\|uv-u'v'\|_{\delta',s+1}+c_{s+1}R_{s+1}^{2}\|u-u'\|_{\delta',s+1}\nonumber \\
 & \le2c_{s+1}R_{s+1}^{2}\|(u,v)-(u',v')\|_{\delta,s+1}\nonumber \\
 & \le\frac{2c_{s+1}R_{s+1}^{2}}{\delta-\delta'}\|(u,v)-(u',v')\|_{\delta,s+1}.\nonumber 
\end{align}
Next, combining 
\[
\mu(u)(uv)_{x}-\mu(u')(u'v')_{x}=\mu(u)\left[(uv)_{x}-(u'v')_{x}\right]+\left[\mu(u)-\mu(u')\right](u'v')_{x}
\]
with  \eqref{eq:uv-1} and  \eqref{eq:u'v'_x}, we obtain 
\begin{align}
 & \|\mu(u)(uv)_{x}-\mu(u')(u'v')_{x}\|_{\delta',s+1}\label{eq:mu u uv-1}\\
 & \le|\mu(u)|\|(uv)_{x}-(u'v')_{x}\|_{\delta',s+1}+|\mu(u)-\mu(u')|\|(u'v')_{x}\|_{\delta',s+1}\nonumber \\
 & \le\frac{2e^{-1}c_{s+1}R_{s+1}^{2}}{\delta-\delta'}\|(u,v)-(u',v')\|_{\delta,s+1}.\nonumber 
\end{align}

The next step is to consider $-v^{3}/3$, $-(v^{3})_{x}/3$ and $\mu(v^{3})$.
 The factorization $v^{3}-v'^{3}=(v^{2}+vv'+v'^{2})(v-v')$ implies
\[
\|v^{3}-v'^{3}\|_{\delta,s+1}\le c_{s+1}\|v^{2}+vv'+v'^{2}\|_{\delta,s+1}\|v-v'\|_{\delta,s+1}.
\]

Here we have 
\[
\|v^{2}+vv'+v'\|_{\delta,s+1}^{2}\le c_{s+1}\left(\|v\|_{\delta,s+1}^{2}+\|v\|_{\delta,s+1}\|v'\|_{\delta,s+1}+\|v'\|_{\delta,s+1}^{2}\right)\le3c_{s+1}R_{s+1}^{2}.
\]
Therefore 
\begin{equation}
\|v^{3}-v'^{3}\|_{\delta,s+1}\le3c_{s+1}^{2}R_{s+1}^{2}\|v-v'\|_{\delta,s+1},\label{eq:v^3}
\end{equation}
and it immediately gives
\begin{equation}
\|v^{3}-v'^{3}\|_{\delta,s+1}\le\frac{3c_{s+1}^{2}R_{s+1}^{2}}{\delta-\delta'}\|v-v'\|_{\delta,s+1}.\label{eq:v^3-2}
\end{equation}
Combining \eqref{eq:v^3} with Proposition \ref{prop:partial_x},
we get 
\begin{align}
\|(v^{3})_{x}-(v'^{3})_{x}\|_{\delta',s+1} & \le\frac{e^{-1}}{\delta-\delta'}\|v^{3}-v'^{3}\|_{\delta,s+1}\label{eq:v^3-1}\\
 & \le\frac{3e^{-1}c_{s+1}^{2}R_{s+1}^{2}}{\delta-\delta'}\|v-v'\|_{\delta,s+1}.\nonumber 
\end{align}
An immediate consequence of \eqref{eq:v^3-2} is 
\begin{equation}
\|\mu(v^{3})-\mu(v'^{3})\|_{\delta',s+1}=\|v^{3}-v'^{3}\|_{\delta',s+1}\le\frac{3c_{s+1}^{2}R_{s+1}^{2}}{\delta-\delta'}\|v-v'\|_{\delta,s+1}.\label{eq:mu v 3}
\end{equation}
The last step is to consider the images of $\mu^{2}(u)u,\mu^{2}(u)v,\mu(u)v^{2}$
and $\mu(u)(v^{2})_{x}$ under $\partial_{x}A^{-1}$. Since 
\begin{align*}
\mu^{2}(u)u-\mu^{2}(u')u' & =\mu^{2}(u)(u-u')+\left[\mu(u)+\mu(u')\right]\left[\mu(u)-\mu(u')\right]u',\\
\mu^{2}(u)v-\mu^{2}(u')v' & =\mu^{2}(u)(v-v')+\left[\mu(u)+\mu(u')\right]\left[\mu(u)-\mu(u')\right]v',\\
\mu(u)v^{2}-\mu(u')v'^{2} & =\mu(u)(v^{2}-v'^{2})+\left[\mu(u)-\mu(u')\right]v'^{2},\\
\mu(u)(v^{2})_{x}-\mu(u')(v'^{2})_{x} & =\mu(u)(v^{2}-v'^{2})_{x}+\left[\mu(u)-\mu(u')\right](v'^{2})_{x},
\end{align*}
we have, by $c_{s}\le c_{s+1}$, $\|\cdot\|_{\delta,s}\le\|\cdot\|_{\delta,s+1}$
and $R_{s}\le R_{s+1}$, 
\begin{align*}
\|\mu^{2}(u)u-\mu^{2}(u')u'\|_{\delta,s} & \le|\mu^{2}(u)|\|u-u'\|_{\delta,s}+\left|\mu(u)+\mu(u')\right|\left|\mu(u-u')\right|\|u'\|_{\delta,s}\\
 & \le3R_{s+1}^{2}\|u-u'\|_{\delta,s+1},\\
\|\mu^{2}(u)v-\mu^{2}(u')v'\|_{\delta,s} & \le|\mu^{2}(u)|\|v-v'\|_{\delta,s}+\left|\mu(u)+\mu(u')\right|\left|\mu(u-u')\right|\|v'\|_{\delta,s}\\
 & \le3R_{s+1}^{2}\|(u,v)-(u',v')\|_{\delta,s+1}\\
\|\mu(u)v^{2}-\mu(u')v'^{2}\|_{\delta,s} & \le|\mu(u)|\|v^{2}-v'^{2}\|_{\delta,s}+\left|\mu(u-u')\right|\|v'^{2}\|_{\delta,s}\\
 & \le3c_{s+1}R_{s+1}^{2}\|(u,v)-(u',v')\|_{\delta,s+1},\\
\|\mu(u)(v^{2})_{x}-\mu(u')(v'^{2})_{x}\|_{\delta,s} & \le2\pi|\mu(u)|\|v^{2}-v'^{2}\|_{\delta,s+1}+2\pi\left|\mu(u-u')\right|\|v'^{2}\|_{\delta,s+1}\\
 & \le6\pi c_{s+1}R_{s+1}^{2}\|(u,v)-(u',v')\|_{\delta,s+1}.
\end{align*}
To derive the last one, we have used the second inequality of Proposition
\ref{prop:partial_x}. We employ  \eqref{eq:partialx A-1} to obtain
\begin{align}
 & \left\Vert \partial_{x}A^{-1}\left[\mu^{2}(u)u-\mu^{2}(u')u'\right]\right\Vert _{\delta',s+1}\le\frac{3e^{-1}R_{s+1}^{2}}{\delta-\delta'}\|u-u'\|_{\delta,s+1},\label{eq:719a}\\
 & \left\Vert \partial_{x}A^{-1}\left[\mu^{2}(u)v-\mu^{2}(u')v'\right]\right\Vert _{\delta',s+1}\le\frac{3e^{-1}R_{s+1}^{2}}{\delta-\delta'}\|(u,v)-(u',v')\|_{\delta,s+1},\label{eq:719b}\\
 & \left\Vert \partial_{x}A^{-1}\left[\mu(u)v^{2}-\mu(u')v'^{2}\right]\right\Vert _{\delta',s+1}\label{eq:719c}\\
 & \le\frac{3e^{-1}c_{s+1}R_{s+1}^{2}}{\delta-\delta'}\|(u,v)-(u',v')\|_{\delta,s+1},\nonumber \\
 & \left\Vert \partial_{x}A^{-1}\left[\mu(u)(v^{2})_{x}-\mu(u')(v'^{2})_{x}\right]\right\Vert _{\delta',s+1}\label{eq:719d}\\
 & \le\frac{6\pi e^{-1}c_{s+1}R_{s+1}^{2}}{\delta-\delta'}\|(u,v)-(u',v')\|_{\delta,s+1}.\nonumber 
\end{align}
It is much easier to obtain 
\begin{align}
\left\Vert \partial_{x}A^{-1}\left[\gamma u-\gamma u'\right]\right\Vert _{\delta',s+1} & \le\frac{e^{-1}\gamma}{\delta-\delta'}\|u-u'\|_{\delta,s+1},\label{eq:gamma u}\\
\left\Vert \partial_{x}A^{-1}\left[\gamma v-\gamma v'\right]\right\Vert _{\delta',s+1} & \le\frac{e^{-1}\gamma}{\delta-\delta'}\|v-v'\|_{\delta,s+1}.\label{eq:gamma v}
\end{align}

Set
\begin{align*}
 & F_{\mu,1}(u,v)=-2\mu(u)uv+\dfrac{1}{3}v^{3}-\partial_{x}A^{-1}\left[2\mu^{2}(u)u+\mu(u)v^{2}+\gamma u\right]-\dfrac{1}{3}\mu(v^{3}),\\
 & F_{\mu,2}(u,v)=-2\mu(u)(uv)_{x}+\dfrac{1}{3}(v^{3})_{x}-\partial_{x}A^{-1}\left[2\mu^{2}(u)v+\mu(u)(v^{2})_{x}+\gamma v\right].
\end{align*}
Then by \eqref{eq:mu u uv}, \eqref{eq:v^3}, \eqref{eq:719a}, \eqref{eq:719c},
\eqref{eq:mu v 3} and \eqref{eq:gamma u}, we have 
\begin{align}
 & \|F_{\mu,1}(u,v)-F_{\mu,1}(u',v')\|_{\delta',s+1}\label{eq:LipF1}\\
 & \le\frac{R_{s+1}^{2}(4c_{s+1}+2c_{s+1}^{2}+6e^{-1}+3e^{-1}c_{s+1})+e^{-1}\gamma}{\delta-\delta'}\|(u,v)-(u',v')\|_{\delta,s+1}.\nonumber 
\end{align}
By \eqref{eq:mu u uv-1}, \eqref{eq:v^3-1}, \eqref{eq:719b}, \eqref{eq:719d}
and \eqref{eq:gamma v}, we have 
\begin{align}
 & \|F_{\mu,2}(u,v)-F_{\mu,2}(u',v')\|_{\delta',s+1}\label{eq:LipF2}\\
 & \le\frac{e^{-1}R_{s+1}^{2}(4c_{s+1}+c_{s+1}^{2}+6+6\pi c_{s+1})+e^{-1}\gamma}{\delta-\delta'}\|(u,v)-(u',v')\|_{\delta,s+1}.\nonumber 
\end{align}
We have obtained two inequalities of the Lipschitz type.

Next we estimate $F_{j}(u_{0},v_{0})\,(j=1,2)$. We have 
\begin{align*}
\|\mu(u_{0})u_{0}v_{0}\|_{\delta,s+1} & \le c_{s+1}\|u_{0}\|_{1,s+1}^{2}\|v_{0}\|_{1,s+1},\\
\|v_{0}^{3}\|_{\delta,s+1} & \le c_{s+1}^{2}\|v_{0}\|_{1,s+1}^{3},\\
\|\partial_{x}A^{-1}\left[\mu^{2}(u_{0})u_{0}\right]\|_{\delta,s+1} & \le\frac{e^{-1}}{1-\delta}\|u_{0}\|_{1,s+1}^{3},\\
\|\partial_{x}A^{-1}\left[\mu(u_{0})v_{0}^{2}\right]\|_{\delta,s+1} & \le\frac{e^{-1}c_{s+1}}{1-\delta}\|u_{0}\|_{1,s+1}\|v_{0}\|_{1,s+1}^{2},\\
\|\mu(v_{0}^{3})\|_{\delta,s+1} & \le\|v_{0}\|_{1,s+1}^{3},\\
\|\partial_{x}A^{-1}[\gamma u_{0}]\|_{\delta,s+1} & \le\frac{e^{-1}\gamma}{1-\delta}\|u_{0}\|_{1,s+1}.
\end{align*}
Since $X^{2}Y\le(X+Y)^{3}/3,X^{3}\le(X+Y)^{3}$ for $X,Y\ge0$, we
get 
\begin{align}
 & \|F_{\mu,1}(u_{0},v_{0})\|_{\delta,s+1}\\
 & \le\frac{2c_{s+1}+c_{s+1}^{2}+6e^{-1}+e^{-1}c_{s+1}+1}{3(1-\delta)}\|(u_{0},v_{0})\|_{1,s+1}^{3}\nonumber \\
 & \quad+\frac{e^{-1}\gamma}{1-\delta}\|(u_{0},v_{0})\|_{1,s+1}.\nonumber 
\end{align}
Similarly we have 
\begin{align*}
\|\mu(u_{0})(u_{0}v_{0})_{x}\|_{\delta,s+1} & \le\frac{e^{-1}c_{s+1}}{1-\delta}\|u_{0}\|_{1,s+1}^{2}\|v_{0}\|_{1,s+1},\\
\|(v_{0}^{3})_{x}\|_{\delta,s+1} & \le\frac{e^{-1}c_{s+1}^{2}}{1-\delta}\|v_{0}\|_{1,s+1}^{3},\\
\|\partial_{x}A^{-1}\left[\mu^{2}(u_{0})v_{0}\right]\|_{\delta,s+1} & \le\frac{e^{-1}}{1-\delta}\|u_{0}\|_{1,s+1}^{2}\|v_{0}\|_{1,s+1},\\
\|\partial_{x}A^{-1}\left[\mu(u_{0})(v_{0}^{2})_{x}\right]\|_{\delta,s+1} & \le\frac{2\pi e^{-1}c_{s+1}}{1-\delta}\|u_{0}\|_{1,s+1}\|v_{0}\|_{1,s+1}^{2},\\
\|\partial_{x}A^{-1}[\gamma v_{0}]\|_{\delta,s+1} & \le\frac{e^{-1}\gamma}{1-\delta}\|v_{0}\|_{1,s+1}.
\end{align*}
Therefore 
\begin{align}
 & \|F_{\mu,2}(u_{0},v_{0})\|_{\delta,s+1}\\
 & \le\frac{e^{-1}(2c_{s+1}+c_{s+1}^{2}+2+2\pi c_{s+1})}{3(1-\delta)}\|(u_{0},v_{0})\|_{1,s+1}^{3}\nonumber \\
 & \quad+\frac{e^{-1}\gamma}{1-\delta}\|(u_{0},v_{0})\|_{1,s+1}.\nonumber 
\end{align}
\end{proof}
We set $R=\|(u_{0},v_{0})\|_{1,s+1}$. If $\gamma=0$, the constants
corresponding to $L$ and $M$ in \eqref{eq:Ov Fu0} and \eqref{eq:Ov Lipschitz}
are of degrees 2 and 3 respectively. Therefore $T$ equals a constant
multiple of $\|(u_{0},v_{0})\|_{1,s+1}^{-2}$ if $\gamma=0$. If $\gamma\ne0$,
we have to consider two cases separately: large or small initial values.
When the initial values are large, larger order terms are dominant
and $T$ is approximated by a constant multiple of $\|(u_{0},v_{0})\|_{1,s+1}^{-2}$,
while $T$ approaches a constant as the initial values approach 0.
Theorem \ref{thm:equivalence} allows us to get a result about the
original Cauchy problem \eqref{eq:modifiedmuCH-1}. We assume $u_{0}\in G^{\Delta,s+2}$
so that $v_{0}=\partial_{x}u_{0}$ belongs to $G^{\Delta,s+1}$.
\begin{thm}
If $u_{0}$ is a real-analytic function on $S^{1}$, then the Cauchy
problem  \eqref{eq:modifiedmuCH-1} has a holomorphic solution near
$t=0$. More precisely, we have the following:

(i) There exists $\Delta>0$ such that $u_{0}\in G^{\Delta,s+2}$
for any $s$.

(ii) If $s>1/2$, there exists a positive time $T_{\Delta}=T(u_{0},s,\Delta)$
such that for every $d\in(0,1)$, the Cauchy problem \eqref{eq:modifiedmuCH-1}
has a unique solution which is a holomorphic function valued in $G^{\Delta d,s+2}$
in the disk $D(0,T_{\Delta}(1-d))$. Furthermore, if $\gamma=0$,
the analytic lifespan $T_{\Delta}$ satisfies

\[
T_{\Delta}=\frac{\mathrm{const.}}{\|(u_{0},u_{0}')\|_{\Delta,s+1}^{2}},
\]
when $\Delta$ is fixed. On the other hand, if $\gamma\ne0$, we have
the asymptotic behavior
\[
T_{\Delta}\approx\frac{\mathrm{const.}}{\|(u_{0},u_{0}')\|_{\Delta,s+1}^{2}}\;(\text{large initial values}),\quad T_{\Delta}\approx\mathrm{const.}\;(\text{small initial values}).
\]
\end{thm}

\begin{rem}
In \cite{FuYing}, the author solved \eqref{eq:modifiedmuCH-1} in
a space of analytic functions following \cite{HM}. What is new in
the present paper is precise estimates of the lifespan.
\end{rem}

\subsection*{Acknowledgment}

This work was partially supported by JSPS KAKENHI Grant Number 26400127.

\end{document}